\documentclass{article}
\usepackage{amsmath}
\usepackage{amssymb}
\usepackage{amsfonts,color}
\usepackage{hyperref}
\usepackage{lineno}

\newtheorem{theorem}{Theorem}[section]

\newtheorem{corollary}[theorem]{Corollary}

\newtheorem{definition}[theorem]{Definition}

\newtheorem{lemma}[theorem]{Lemma}

\newtheorem{proposition}[theorem]{Proposition}
\newtheorem{remark}[theorem]{Remark}

\newenvironment{proof}[1][Proof]{\noindent\textbf{#1.} }{\ \rule{0.5em}{0.5em}}
\def\R{\mathbb R}

\input{tcilatex}
\begin{document}

\title{Regularity properties of the solution to a stochastic heat equation
driven by a fractional Gaussian noise on ${\mathbb{S}}^2$}
\author{Xiaohong Lan \thanks{%
Research of X. Lan is supported by NSFC grant 11501538 and CAS grant
QYZDB-SSW-SYS009, E-mail: xhlan@ustc.edu.cn } \\
School of Mathematical Sciences \\
University of Science and Technology of China \\
E-mail: xhlan@ustc.edu.cn\\
\and Yimin Xiao \thanks{%
Corresponding author. Research of Y. Xiao is partially supported by grants
 DMS-1612885 and DMS-1607089 from the National Science Foundation.} \\
Department of Statistics and Probability\\
Michigan State University\\
E-mail: xiao@stt.msu.edu}
\maketitle

\begin{abstract}
We study the stochastic heat equation driven
by an additive infinite dimensional fractional Brownian noise on the unit sphere $\mathbb{S}^{2}$. 
The existence and uniqueness of its solution in certain Sobolev space is investigated and 
sample path regularity properties are established. In particular, the exact uniform modulus of
continuity of the solution in time/spatial variable is derived. 
\end{abstract}




\section{Introduction}

For more than three decades, the topic of stochastic partial differential
equations (SPDEs, henceforth) has been an active area of research in pure
and applied mathematics. The study of SPDEs has recently entered a period of
rapid growth. We refer to \cite{DaPraZab, GM11, PrevRoc, Walsh} for
systematic accounts on SPDEs, and to \cite{DKMNX, Kh16, Kh14} for some
recent developments.

In this paper, we consider the following stochastic heat equation driven by
an infinite dimensional fractional Brownian noise $W^{H} = \left\{ W^{H}\left(t\right) ,\, t \in \R_+ \right\} $ 
on the unit sphere $\mathbb{S}^{2}\subset \mathbb{R}^{3}$: 
\begin{equation}
du(t)=\Delta _{{\mathbb{S}}^{2}}u(t)dt+dW^{H}(t)  
\label{def:heat eq}
\end{equation}%
\bigskip with initial condition $u(0)=u_{0}\in L^{2}(\Omega \times {\mathbb{S}}^{2},\mathbb{P}\times \nu)$, where 
$\nu$ is the Lebesgue measure on $\mathbb{S}^{2}$.  

In (\ref{def:heat eq}),  $\Delta _{{\mathbb{S}}^{2}}$  is the 
Laplace-Beltrami operator on $\mathbb{S}^{2}$ defined as 
\begin{equation*}
\Delta _{{\mathbb{S}}^{2}}=\frac{\partial ^{2}}{\partial \vartheta ^{2}}
+\cot \vartheta \frac{\partial ^{2}}{\partial \vartheta ^{2}}+(\sin \vartheta )^{-2}\frac{\partial ^{2}}{\partial \phi ^{2}},
\end{equation*}%
where $\vartheta \in \lbrack 0,\pi ]$ represents the latitude and $\phi \in \lbrack 0,2\pi )$ the longitude 
in spherical coordinates.  The Gaussian noise $W^{H}$ is a specialization of the infinite 
dimensional fractional Brownian noise in Tindel et al (\cite{TinTudVien03}) to the sphere $\mathbb S^2$. More
precisely, it is defined as follows. 

\begin{definition}
\label{def:W^H} 
The noise $W^{H}=\{W^{H}(t),\,t\in \mathbb{R}_{+}\}$ is a Gaussian process with the following representation
\begin{equation}
W^{H}(t,x)=\sum_{\ell \geq 0}\sum_{m=-\ell }^{\ell }\sqrt{C_{\ell }}
\beta_{\ell m}(t)Y_{\ell m}(x),
\label{specW}
\end{equation}%
where $\{C_{\ell }, \ell =0,1,2,...\}$ is a sequence of  positive 
constants,  $Y_{\ell m},\, (\ell =0,1,2, ...,$ $ m=0, \pm 1,\ldots ,\pm \ell)$ are the
spherical harmonic functions on ${\mathbb{S}}^{2}$ satisfying 
\begin{equation*}
\Delta _{{\mathbb{S}}^{2}}Y_{\ell m}=-\ell (\ell +1)Y_{\ell m},
\end{equation*}%
and the sequence of complex-valued Gaussian processes $\left\{ \beta _{\ell m}(t)\right\}
_{\ell m}$ satisfies the following two conditions
\begin{itemize}
\item[(a)]
for every $t\in\mathbb{R}$ and $\ell =0,1,2,..., m=0,\pm 1,\ldots ,\pm \ell $, 
\begin{equation}
\overline{\beta _{\ell m}}(t)=(-1)^{m}\beta _{\ell ,-m}(t).
\label{conj-beta}
\end{equation}%
\item[(b)]
$\{\sqrt{2}Re\beta _{\ell m}(t),\sqrt{2}Im\beta _{\ell m}(t),\ \ell
=0,1,2,...,\ m=0,...,\ell \}$ is a sequence of independent copies of a real-valued
fractional Brownian motion $B^H = \{B^H(t), t \in \R_+\}$ with Hurst index $H\in (0,1)$. 
\end{itemize}
\end{definition}

\begin{remark}
It is readily seen that the sequence of Gaussian processes $\left\{ \beta_{\ell m}(t)\right\} _{\ell m}$ satisfies   
\begin{equation}
\mathbb{E}\big[\beta _{\ell m}(t)\overline{\beta _{\ell ^{\prime }m^{\prime}}}(s)\big]
=\delta _{\ell }^{\ell ^{\prime }}\delta _{m}^{m^{\prime}}R_{H}(t,s),  
\label{indep-beta}
\end{equation}%
for all $\ell =0,1,2,...,\ m=-\ell ,...,\ell $, where $\delta _{\ell }^{\ell^{\prime }} = 1$ if $\ell = \ell'$ 
and 0 otherwise, and where 
\begin{equation*}
R_{H}(t,s)=\frac{1}{2}\big[t^{2H}+s^{2H}-|t-s|^{2H}\big].
\end{equation*}%
\end{remark}

For the positive coefficients $\{C_{\ell }, \ell = 0, 1, 2, \ldots\}$ in \eqref{specW}, we assume   
the following condition:

\textbf{Condition (A.1)} \emph{There  exist constants} $\alpha >0$\emph{\ and }  $c_{0}>1$
\emph{ such that } 
\begin{equation*}
C_{\ell }=\Upsilon \left( \ell \right) (\ell +1/2)^{-\alpha }\quad \hbox{ \it and } \  \ c_{0}^{-1}\leq \Upsilon \left( \ell \right) \leq c_{0}
\end{equation*}%
for $\ell =0,1,2,...$

Note that for $\alpha>2$, $W^{H} = \{ W^{H}(t), \, t \in \mathbb{R}_+ \}$ is  an $L^{2}(\mathbb{S}^2)$-valued 
$\Lambda$-fractional Brownian motion with $\Lambda$ given below in (\ref{def:Lamda}). When $0<\alpha \leq 2$, $ W^{H}(t)$ 
can be viewed as a generalized fractional Brownian motion taking values on some Hilbert space $U 
\supset L^2(\mathbb{S}^2)$. For instance,  $U$ is a Hilbert space such that for any $\varphi ,\psi \in U$, 
\begin{equation} \label{def-Hilbert U}
\left\langle \varphi ,\psi \right\rangle _{U}
=\int_{\mathbb{S}^{2}}\int_{\mathbb{S}^{2}}\varphi (x)\widetilde{\Lambda} (x,y)\psi (y)d\nu (x)d\nu (y),
\end{equation}%
where $\widetilde{\Lambda} (x,y)=\sum_{\ell =1}^{\infty }\ell ^{-2}P_{\ell }\left( \left\langle x,y\right\rangle \right) $ for 
all $x,y\in \mathbb{S}^{2}$.  Here $P_{\ell }: [-1,1]\rightarrow \mathbb{R},\, (\ell =0,1,2,...)$ are the Legendre 
polynomials satisfying the normalization condition $P_{\ell }(1)=1$ for all $\ell $. So even if 
$\sum_{\ell =0}^{\infty }C_{\ell }=\infty$, $W^H(t)$ given by (\ref{specW}) is a well-defined $U$-valued 
Gaussian process, see \cite{TinTudVien03} for more discussion about the Gaussian process  $W^H(t)$ 
taking values on more general Hilbert spaces.  In analogy to the Euclidean space setting considered in 
\cite{BalTud08, XiaoTudor08}, one may refer to $W^{H}$ as a fractional-colored Gaussian noise on the sphere 
$\mathbb S^2$. 


The present paper is mainly motivated by the recent works of Lang and Schwab \cite{LangSchAAP} who 
studied (\ref{def:heat eq}) driven by a $Q$-Wiener process which corresponds to the case of $H=\frac{1}{2}$ 
and $\alpha >2$ in the setting of the present paper, and by Tindel et al. \cite{TinTudVien03, TinTudVien04, NuaV09} 
who studied (\ref{def:heat eq}) with $x\in \mathbb{S}^{1}$, which is the unit circle in 
$\mathbb{R}^{2}$, and the fractional Gaussian noise $W^{H}$ on $\mathbb{S}^{1}$ for an arbitrary $H\in (0,1)$. 
Our objectives are to establish the existence and uniqueness of the mild solution of (\ref{def:heat eq}), 
and to study the regularity properties of the solution process when it exists. For simplicity, we focus in this paper 
on the case of $\frac{1}{2}<H<1.$ The case of $0<H<\frac{1}{2}$ is more delicate and will be considered in a subsequent 
paper.

In the stochastic heat equation (\ref{def:heat eq}), we make some assumptions on the initial value $u_{0}
=\{u_{0}(x),x\in {\mathbb{S}}^{2}\}$. First recall  from \cite{MPbook} that, 
a spherical random field $Z=\{Z(x),x\in\mathbb{S}^{2}\}$ is called 2-weakly isotropic if 
\begin{equation*}
\mathbb{E}[Z(x)Z(y)]=\mathbb{E}[Z(gx)Z(gy)]
\end{equation*}
for all $x,y\in \mathbb{S}^{2}$ and $g\in SO(3)$.  If $u_{0}$ is a zero-mean, 2-weakly isotropic random field 
with finite variance, then by Theorem 5.13 in \cite{MPbook}, we have the following spectral representation: 
\begin{equation}
u_{0}(x)=\sum_{\ell \geq 0}\sum_{m=-\ell }^{\ell }u_{0,\ell m}Y_{\ell m}(x), \ \ \ \hbox{ a.s.,  }
\label{spec-u0}
\end{equation}%
where the random variables $u_{0,\ell m},\, \ell =0,1,2,...,m=0,\pm 1,\ldots,\pm \ell $, satisfy 
\begin{equation*}
\mathbb{E}\big(u_{0,\ell m}\overline{u_{0,\ell ^{\prime }m^{\prime }}}\big)
=\delta _{\ell }^{\ell ^{\prime }}\delta _{m}^{m^{\prime }}D_{\ell }
\end{equation*}%
for some nonnegative constants $D_{\ell },\ \ell =0,1,2,...$. The sequence $\{D_\ell, \ell \ge 0\}$ is called the
angular power spectrum of $u_{0}$. We will make use of  the following assumption.

\textbf{Condition (A.2)} Either $u_{0}\equiv 0$\emph{\ or }
$u_{0}=\left\{u_{0}(x),\ x\in {\mathbb{S}}^{2}\right\} $\emph{\ is a zero-mean 
isotropic Gaussian field which is independent from the Gaussian
noise } $W^{H}$\emph{. Moreover, there
exist finite constants $\beta >4$\, and $D_{0}>0$  such that }%
\begin{equation*}
D_{\ell }\leq D_{0}(\ell +1/2)^{-\beta }
\end{equation*}%
\emph{for all }$\ell =0,1,2,...$

In order to state our main theorem, we introduce the following notations. Let $\mathbb{I}$ be an open interval 
on $\mathbb{R}$.  For a function $u: \mathbb{I} \to \mathbb{R}$ and an integer $k>0$, we say that the $k^{th}$  
weak derivative of $u$ exists if there exists a locally integrable function $v$ such that for all infinitely differentiable
function $\varphi $ with compact support on $\mathbb{I}$, %
\begin{equation*}
\int_{\mathbb{I}}u\,D_{t}^{k}\varphi dt=(-1)^{k}\int_{\mathbb{I}}v\,\varphi dt.
\end{equation*}%
Such kind of function $v$ is uniquely determined up to a zero-measure set on $\mathbb{I}$, and we write 
$v=:D_{t}^{k}u$ as usual. Let $\mathbb{H}^{k}(\mathbb{I})$ be the subspace of $L^{2}(\mathbb{I})$ such that 
\begin{equation*}
\mathbb{H}^{k}(\mathbb{I})=\left\{ u\in L^{2}(\mathbb{I}): \forall m=0,1,..., k,\, D_{t}^{m}u\ \hbox{ exists and belongs to } 
L^{2}(\mathbb{I})\right\} .
\end{equation*}%
The space $\mathbb{H}^{k}(\mathbb{I})$ is also called the Sobolev space with $k^{th}$
weak derivatives having finite $L^{2}$-norm, see for instance \cite{GilTrud,PrevRoc} for 
more details about this and more general Sobolev spaces. 

Throughout this paper,  $T$ is a positive and finite constant and $\mathbb{T}=  \lbrack 0,T]$.
The following theorem is the main result of this paper.

\begin{theorem}
\label{Th:Main} 
Assume $H > 1/2$ and Conditions \textbf{(A.1) } and \textbf{(A.2) } hold. Then eq.(\ref{def:heat eq}) 
has a unique solution $\{u(t,x),t\in \mathbb{T},x\in {\mathbb{S}}^{2}\}$  in $L^{2}(\Omega \times \mathbb 
T\times \mathbb{S}^{2})$ which is a mean-zero Gaussian random field that is 2-weakly isotropic in $x$ (for each fixed 
$t \in \mathbb T$). Moreover, the solution has the following regularity properties:

\begin{enumerate}
\item[(a)] If $\alpha +4H>4$, then for every $t\in \mathbb T,$ we have $u(t,\cdot )\in C^{1}(\mathbb{S}^{2})$
a.s. Moreover, if $\alpha +4H>6$, then for every $t\in \mathbb T$,
$u(t,\cdot )\in C^{2}(\mathbb{S}^{2})$ a.s. and,  for every $x\in \mathbb{S}^{2}$, 
$u(\cdot ,x)\in \mathbb{H}^{1}(\mathbb T)$ a.s.

\item[(b)]  If $u_{0}\equiv 0$ or $\beta >4H+2$, then for every $x\in \mathbb{S}^{2}$, 
\begin{equation}
\lim_{\varepsilon \rightarrow 0}\sup_{\substack{ 0\leq s<t\leq T, \\ t-s\leq
\varepsilon }}\frac{|u(t,x)-u(s,x)|}{(t-s)^{\eta }\sqrt{\big|\log (t-s)\big|}}
=K_{1,1},\ a.s., 
 \label{eq:Modulus-t}
\end{equation}%
where $\eta =H-\max {\{(2-\alpha )/4,0\}}$, and where $K_{1,1}>0$ is a constant depending on 
$c_{0},\ \alpha $ and $H$. 

\item[(c)] If $\alpha +4H<4$ and $u_{0}\equiv 0$, then for every $t\in \mathbb T$, there 
exists a constant $K_{1,2}>0$ depending on $c_{0},\alpha,\, t $ and $H$ such that%
\begin{equation}
\lim_{\varepsilon \rightarrow 0}\sup_{\substack{ x,y\in {\mathbb{S}}^{2} \\ %
d_{{\mathbb{S}}^{2}}(x,y)\leq \varepsilon }}
\frac{|u(t, x)-u(t, y)|}{\left( d_{{\mathbb{S}}^{2}}(x,y)\right) ^{\gamma }
\sqrt{\big|\log d_{{\mathbb{S}}^{2}}(x,y)\big|}}=K_{1,2},\ a.s.,  
\label{eq:Modulus-x}
\end{equation}
where $\gamma =\alpha /2-1+2H\in (0,1)$ and $d_{{\mathbb{S}}^{2}}(x,y)$ denotes the geodesic distance 
between $x, y$.
\end{enumerate}
\end{theorem}

In the following, we give some remarks that compare our results with some existing ones in the literature and 
raise some unsolved problems.

\begin{itemize}
\item Tindel et al. \cite{TinTudVien03} proved the almost-sure H$\ddot{o}$lder continuity of $u$ with respect to ($w.r.t.$) 
the time variable $t$ for $u(t)$ taking values on some general Hilbert space $V$.   Eq. (\ref{eq:Modulus-t}) significantly 
improves their result because it implies the exact uniform modulus of continuity of $u(t)$ taking values on $L^{2}(\mathbb{S}^{2})$.  
Moreover, since $\eta$ increases with $\alpha>0$,   (\ref{eq:Modulus-t}) indicates that the H$\ddot{o}$lder continuity of the solution in 
the time variable $t$ can be improved by the smoothness of spatial 
structure in the fractional-colored noise $W^{H}$. 

\item Parts (a) and (c)  show that the smoothness of  the solution in  the spatial variable $x$ deteriorates when $\alpha$ decreases. 
In particular, when $2<\alpha+4H<4$, $u$ is not differentiable $w.r.t.$ spatial variable $x$. Eq. (\ref{eq:Modulus-x}) provides the 
exact uniform  modulus of continuity of $u$ in $x$. On the other hand, when $\alpha+4H>4$,  we are able to establish more precise 
information on the smoothness of $u$ in $x$  by providing the exact uniform moduli of continuity of higher-order derivatives of $u$ in 
$x\in \mathbb{S}^{2}$. See Corollary \ref{Cor:derivative-x} below.          

\item Note that we have left out  the cases of $4<\beta \leq 4H+2$ and $\alpha+4H=4$. In the first case, the regularity properties 
of $u_0$ may compete with those of the solution of eq.(\ref{def:heat eq}) with  $u_{0}\equiv 0$. In the second case,   
we have not been able to prove the strong local nondeterminism for the solution $u$. These cases are more subtle and some new method
may be needed.
\end{itemize}

The plan of this paper is as follows. In Section \ref{Sec:fractional-colored noise}, we recall some basic properties of
stochastic integration with respect to fractional Brownian motion and investigate the smoothness of the Gaussian noise $W^{H}$. 
We present the unique existence of the mild solution $u(t)$ in $L^{2}(\Omega\times\mathbb{S}^{2})$ in Section \ref{Sec:Existence}. 
In particular, we give the uniform convergence of this mild solution when $\alpha+4H>4$, which leads to that $u$ also exists in 
some nice Sobolev spaces. In Section \ref{Sec: TechnicalTools}, some auxiliary technical tools such as estimation of variogram and 
strong local nondeterminism of the solution $\{u(t,x),t\in \mathbb{T},x\in {\mathbb{S}}^{2}\}$ are provided. These are not only instrumental 
for our proofs in this paper but also useful for other purposes (cf. \cite{Xiao(SLND)}).
Finally, we prove the exact moduli of continuity (\ref{eq:Modulus-t}) and (\ref{eq:Modulus-x}) of the solution $u$ in 
Section \ref{Sec:ModulusContinuity}. 

Throughout this paper, we denote by $``A\approx B"$ the commensurate case that there
exist positive and finite constants $c_{1}<c_{2}$ such that $c_{1}B\leq A\leq c_{2}B$.



\section{The Fractional-colored Noise \label{Sec:fractional-colored noise}}

In this section, we provide some preliminaries on the fractional Gaussian noise $W^{H}\left( t\right).$ 
Moreover, when $\alpha>2$, we establish some regularity properties of the random field $\{W^{H}( t,x),$ 
$ t \in \mathbb{R}_+, x \in 
{\mathbb{S}}^2\}$. 

We first recall briefly some well-known results about  stochastic integration with respect to fractional 
Brownian motion $B^{H}=\{B^{H}(t),\ t\in \mathbb{R}_+\}$  with $H>\frac{1}{2}.$ For any $T>0$, let 
$\mathcal{H}(\mathbb{T})$ be the completion of the space linearly spanned by the indicator functions 
$\left\{ \mathrm{1\hskip-2.9trueptl}_{[s,t]},0\leq s<t\leq T\right\} $ with respect to the inner product
\begin{equation*}
\left \langle \mathrm{1\hskip-2.9trueptl}_{[s_{1},t_{1}]},\mathrm{1\hskip-2.9trueptl}_{[s_{2},t_{2}]}
\right\rangle_{\mathcal{H}}
=\int_{s_{1}}^{t_{1}}\int_{s_{2}}^{t_{2}}|s-t|^{2H-2}dsdt,
\end{equation*}%
where $\ 0\leq s_{i}<t_{i}\leq T,\,i=1,2.$ Then, for any $\varphi  \in \mathcal{H}(\mathbb T)$,
the stochastic integral
\begin{equation*}
B^{H}(\varphi )=\int_{0}^{T}\varphi \left( t\right) dB^{H}(t)
\end{equation*}%
is well-defined. 
Moreover, for any $\varphi ,\psi \in \mathcal{H}(\mathbb T) $,
\begin{equation}
\mathbb{E}[B^{H}(\varphi )B^{H}(\psi )]
=\left\langle \varphi ,\psi \right\rangle _{\mathcal{H}}
=\int_{0}^{T}\int_{0}^{T}\varphi (t)\psi(s)|t-s|^{2H-2}dtds.
\label{fBm-inprod}
\end{equation}%
See for instance  \cite{AMN} or \cite{TinTudVien03} for more details about the stochastic integration 
with respect to fBm $B^{H}$.

In the meantime, it is readily seen by \cite{MPbook}  that when $\alpha
>2,$ the definition of $W^{H}$ in Section 1 is equivalent to that the noise 
$W^{H}=\{W^{H}(t),\,t\in \mathbb{R}_{+}\}$ is an $L^{2}(\mathbb{S}^{2})$%
-valued Gaussian field such that 
$\mathbb{E}\big(W^{H}(t,x)\big)=0$ for all $(t,x)\in \mathbb{R}_{+}\times 
\mathbb{S}^{2}$ and its covariance function is given by 
\begin{equation}
\mathbb{E}\big[W^{H}(t,x)W^{H}(s,y)\big]=R_{H}(t,s)\Lambda (x,y),
\label{Eq:cov}
\end{equation}%
with the spatial covariance function $\Lambda :\mathbb{S}^{2}\times \mathbb{S}^{2}\rightarrow 
\mathbb{R}_{+}$ which can be decomposed into 
\begin{equation}\label{def:Lamda}
\Lambda (x,y)=\sum_{\ell =0}^{\infty }\frac{2\ell +1}{4\pi }C_{\ell }P_{\ell }(\left\langle x,y\right\rangle ). 
\end{equation}%
Notice that (\ref{Eq:cov}) and (\ref{def:Lamda}) imply that 
\begin{equation*}
{\mathbb{E}}\big[W^{H}(t,x)W^{H}(s,y)\big]={\mathbb{E}}\big[W^{H}(t,gx)W^{H}(s,gy)\big]
\end{equation*}%
for every pair of $(t,s)\in \mathbb{R}_{+}^{2},$ 
all $g\in SO(3)$ and all $x,y\in \mathbb{S}^{2}.$ Therefore, for every $t\in \mathbb{R}_{+}$, $W^{H}(t)$ is 
a 2-weakly isotropic random field on ${\mathbb{S}}^{2}$. 

The following proposition provides some properties of the fractional-colored noise $W^{H}$ which will be
exploited later. 

\begin{proposition}
\label{Prop: W^H} 
Under Conditions \textbf{(A.1) }and\textbf{\ (A.2)} with $\alpha >2$, the Gaussian random 
field $W^{H}$ defined in (\ref{specW}) is zero-mean, 2-weakly isotropic and converges in $L^{2}(\Omega )$
uniformly  for $(t,x)\in \mathbb{T}\times {\mathbb{S}}^{2}$. Moreover, the following statements hold:
\begin{enumerate}
\item[(i)] If $1/2<H<1$, then for every $x\in {\mathbb{S}}^{2}$, $W^{H}(\cdot ,x)\in \mathbb{H}^{1}(\mathbb{T})$ a.s.
\item[(ii)] If $\alpha >4,$ then for every $t\in \mathbb{T}$, $W^{H}(t,\cdot )\in C^{1}({\mathbb{S}}^{2})$\ a.s. 
\end{enumerate}
\end{proposition}


\begin{proof}
Based on the previous discussion, we only need to prove (i) and (ii). First
notice that the weak derivatives $\nabla _{{\mathbb{S}}^{2}}W^{H}$ and $D_{t}W^{H}$\ exist due to the uniform 
convergence of $W^{H}$ in $L^{2}(\Omega )$. For part (i), recall    
\begin{equation} \label{SumYlm^2}
\sum_{m=-\ell }^{\ell }Y_{\ell m}(x)\overline{Y_{\ell m}}(y)=\frac{2\ell +1}{4\pi }P_{\ell }(\left\langle x,y\right\rangle ) 
\end{equation}%
(c.f. \cite{MPbook}, section 3.4.2), then careful calculations show that

\begin{equation*}
\begin{split}
& \mathbb{E}\int_{0}^{T}\left\vert D_{t}W^{H}(t,x)\right\vert^{2}dt
=\int_{0}^{T}\mathbb{E}\Big|\sum_{\ell \geq 0}\sum_{m=-\ell }^{\ell }
\sqrt{C_{\ell }}\frac{d}{dt}\beta _{\ell m}(t)Y_{\ell m}(x)\Big|^{2}dt \\
& =\sum_{\ell \geq 0} \frac{2\ell +1}{4\pi }C_{\ell }\int_{0}^{T}\frac{d^{2}}{dt^{2}}
\mathbb{E}|\beta _{\ell m}(t)|^{2}\,dt \\
& =\sum_{\ell \geq 0} \frac{2\ell +1}{4\pi }C_{\ell }\int_{0}^{T}t^{2H-2}dt
=c_{2,1}T^{2H-1}
\end{split}%
\end{equation*}%
for some constant $c_{2,1}>0$ depending on $H.$ The last equality holds if
and only if $\alpha >2$ and $H>\frac{1}{2}$.

For part (ii), recall the gradient $\nabla_{\mathbb{S}^{2}}$ on $\mathbb{S}^{2}$ defined as follows: 
\begin{equation*}
\nabla _{\mathbb{S}^{2}}= \Big(\frac{\partial }{\partial \vartheta },(\sin \vartheta )^{-1}\frac{\partial }{\partial \phi }\Big).
\end{equation*}
For any $x,y\in {\mathbb{S}}^{2}$, denote by $x=\left( \vartheta _{x},\phi _{x}\right) $ and
$\ y=\left( \vartheta _{y},\phi _{y}\right)$. Then careful calculations show that 
\begin{equation*}
\frac{\partial ^{2}}{\partial \vartheta _{y}\partial \vartheta {x}}
P_{\ell }(\left\langle x,y\right\rangle )\left\vert _{x=y}\right. 
=P_{\ell }^{\prime}(1)=-\frac{\ell (\ell +1)}{2},
\end{equation*}%
and 
\begin{equation*}
\frac{\partial ^{2}}{\partial \phi _{y}\partial \phi _{x}}P_{\ell }(\left\langle x,y\right\rangle )\left\vert _{x=y}\right. 
=(\sin \vartheta )^{2}P_{\ell }^{\prime }(1)=-(\sin \vartheta )^{2}\frac{\ell (\ell +1)}{2}.
\end{equation*}%
Therefore, we obtain 
\begin{equation*}
\begin{split}
& \mathbb{E} \Big|\frac{\partial }{\partial \vartheta }W^{H}(t,\vartheta ,\phi )\Big|^{2}
=\left\vert \frac{\partial }{\partial \vartheta _{x}}\frac{\partial }{\partial \vartheta _{y}}
\mathbb{E}\left( W^{H}(t,x)\overline{W^{H}}(t,y)\right) \big\vert _{x=y} \right\vert  \\
& =\bigg\vert \sum_{\ell \geq 0}
C_{\ell }\frac{2\ell +1}{4\pi }
\mathbb{E}|\beta _{\ell m}(t)|^{2}\left\{ \frac{\partial ^{2}}{\partial \vartheta _{y}\partial \vartheta _{x}}
P_{\ell }(\left\langle x,y\right\rangle )\big\vert _{x=y} \right\} \bigg\vert  \\
& =t^{2H}\sum_{\ell \geq 0}
C_{\ell }\frac{2\ell +1} {4\pi }\cdot \frac{\ell (\ell +1)}{2} \\
& \leq \frac{c_{0}}{4\pi }t^{2H}\sum_{\ell \geq 0}\sum_{m=-\ell }^{\ell }(\ell +\frac{1}{2})^{3-\alpha }.
\end{split}%
\end{equation*}%
The last term above is finite if and only if $\alpha >4$. Similarly, we have 
\begin{equation*}
\mathbb{E} \Big|\frac{\partial }{\partial \phi }W^{H}(t,\vartheta ,\phi )\Big|^{2}
\leq \frac{(\sin \vartheta )^{2}}{4\pi }t^{2H}c_{0}\sum_{\ell \geq 0}\sum_{m=-\ell }^{\ell }(\ell +\frac{1}{2})^{3-\alpha }
<\infty ,
\end{equation*}%
and hence
\begin{equation*}
\begin{split}
&\mathbb{E} \Big|\nabla _{{\mathbb{S}}^{2}}W^{H}(t,\vartheta ,\phi )\Big|^{2} \\
&\leq\mathbb{E}\Big|\frac{\partial }{\partial \vartheta }W^{H}(t,\vartheta,\phi )\Big|^{2}
+\frac{1}{(\sin \vartheta )^{2}}\mathbb{E} \Big|\frac{\partial }{\partial \phi }W^{H}(t,\vartheta ,\phi )\Big|^{2}
<\infty 
\end{split}
\end{equation*}
for $\alpha >4$. The proof is then completed.
\end{proof}

\section{Existence of the Solution and Proof of (a) in Theorem \ref{Th:Main} \label{Sec:Existence}}

Recall the form (\ref{specW}) of $W^{H}$, the following is an immediate consequence of Theorem 1 
of Tindel et al. \cite{TinTudVien03}. 

\begin{proposition}
\label{Prop:general exist} 
Assume that $H > 1/2$ and Conditions \textbf{(A.1)} and \textbf{(A.2)} hold. Then
there exists a unique solution $u\in L^{2}(\Omega \times \mathbb{T}\times \mathbb{S}^{2})$ for 
eq.(\ref{def:heat eq}) with the following mild form\ 
\begin{equation} \label{mild-form}
u(t,x)= \sum_{\ell \geq 0}\sum_{m=-\ell }^{\ell }\bigg(e^{-\ell (\ell +1)t}u_{0,\ell m}+
\sqrt{C_{\ell }}\int_{0}^{t}e^{-\ell (\ell +1)(t-s )}d\beta_{\ell m}(s )\bigg)Y_{\ell m}(x)
\end{equation}%
for $(t,x)\in \mathbb{T\times S}^{2}$.  
\end{proposition}

\begin{proof}
First we take the spherical harmonics $\{Y_{\ell m}\}_{\ell m}$ as the orthonormal basis of $L^{2}(\mathbb{S}^{2})$ 
and define $\Phi: L^{2}(\mathbb{S}^{2})\to L^{2}(\mathbb{S}^{2}) $ a positive linear operator such that
 \begin{equation*}
\Phi(Y_{\ell m})=\sqrt{C_{\ell}}Y_{\ell m}.  
\end{equation*}
It is readily seen that the adjoint operator $\Phi^{*}=\Phi$. In the meantime, let $\mathcal{B}^H$ be the cylindrical 
fractional Brownian motion on the Hilbert space $L^{2}(\mathbb{S}^{2})$ defined as 
\begin{equation*}
\mathcal{B}^H(t)=\sum_{\ell \geq 0}\sum_{m=-\ell }^{\ell } \beta_{\ell m}(t) Y_{\ell m},
\end{equation*}
where $\{\beta_{\ell m}(t)\}_{\ell, m}$ is the sequence of complex-valued fBm$s$ in Definition \ref{def:W^H}. Then, 
 \begin{equation*}
\Phi \mathcal{B}^H(t)=:\sum_{\ell \geq 0}\sum_{m=-\ell }^{\ell } \beta_{\ell m}(t) \Phi Y_{\ell m}=W^{H}(t)
\end{equation*}
is well defined in the Hilbert space $U$ with inner product defined in (\ref{def-Hilbert U}).  See \cite[p.190]{TinTudVien03}  
for more discussion about the stochastic integration  $w. r. t.$ $d (\Phi \mathcal{B}^H(t))=\Phi d\mathcal{B}^H(t)$. 

Now we are ready to prove Proposition \ref{Prop:general exist}. It is known from \cite{TinTudVien03} that the mild 
solution of eq.(\ref{def:heat eq})  is unique if it exists and can be written as 
\begin{equation} \label{mild-form0}
u(t)=e^{t\Delta_{\mathbb{S}^{2}}}u_0+\int_{0}^{t}e^{(t-s)\Delta_{\mathbb{S}^{2}}}\Phi d\mathcal{B}^H(s),\ \  t\in [0,T].
\end{equation}
To prove the existence, we only need to establish that the $u(t)$ defined in (\ref{mild-form0}) belongs to 
$L^{2}(\Omega \times \mathbb{S}^{2}) $ for every $t\in [0,T]$.  It is readily seen that under Condition \textbf{(A.2)}, 
\begin{equation*}
e^{t\Delta_{\mathbb{S}^{2}}}u_0= \sum_{\ell \geq 0}\sum_{m=-\ell }^{\ell }e^{-\ell (\ell +1)t}u_{0,\ell m} Y_{\ell m} \in L^{2}(\Omega \times \mathbb{S}^{2}) . 
\end{equation*}  
For the second part on the right hand side of equation (\ref{mild-form0}), we first set the function $G_{H}(\lambda)=( \max {\{\lambda,1\}})^{-2H}$.
 A simple calculation yields that under the spherical harmonic basis $\{Y_{\ell m}\}_{\ell, m}$, the trace of the operator $\Phi^{*}G_{H}(-\Delta_{\mathbb{S}^{2}})\Phi$
\begin{equation*}
\begin{split}
&Tr(\Phi^{*}G_{H}(-\Delta_{\mathbb{S}^{2}})\Phi)=\sum_{\ell \geq 0}\sum_{m=-\ell }^{\ell }\left\langle \Phi^{*}G_{H}(-\Delta_{\mathbb{S}^{2}})
\Phi Y_{\ell m}, Y_{\ell m}\right\rangle_{L^{2}(\mathbb{S}^{2})} \\
&= C_{0}+\sum_{\ell \geq 1}  \frac{2 \ell +1}{4\pi} C_{\ell} (\ell (\ell +1))^{2H} \\
&\leq C_{0}+c_{0} \sum_{\ell \geq 1} \frac{2 \ell +1}{4\pi} (\ell+1/2)^{-\alpha} (\ell (\ell +1))^{-2H}  < \infty
\end{split}
\end{equation*}   
under conditions \textbf{(A.1)} and $H>1/2$. Hence, $\Phi^{*}G_{H}(-\Delta_{\mathbb{S}^{2}})\Phi$ is in the trace class by the fact that this 
operator is positive. Therefore, the unique existence of the mild solution (\ref{mild-form}) to eq.(\ref{def:heat eq}) is obtained by Theorem 1 
in \cite{TinTudVien03}. 
\end{proof}


We shall study in more details  the sample path 
properties of the solution $u(t, x)$ in (\ref{mild-form}) in the following sections. At this moment, 
let us first focus on the special case $H>1/2$ and $\alpha >2$. 
We write $u(t,x)$ in (\ref{mild-form}) in the following form: 
\begin{equation} \label{def:Solution}
u(t,x)=\sum_{\ell =0}^{\infty }u_{\ell }(t,x)=\sum_{\ell =0}^{\infty
}\sum_{m=-\ell }^{\ell }u_{\ell m}(t)Y_{\ell m}(x),
\end{equation}%
where 
\begin{equation}
u_{\ell m}(t)=e^{-\ell (\ell +1)t}u_{0,\ell m}+\sqrt{C_{\ell }}%
\int_{0}^{t}e^{-\ell (\ell +1)(t-\lambda )}d\beta _{\ell m}(\lambda ).
\label{coef-u(t)}
\end{equation}%
Note that under Conditions \textbf{(A.1) }and\textbf{\ (A.2)}, $u_{0}$ and $W^{H}$ are independent, 
which implies that the two sequences $\{u_{0,\ell m}\}_{\ell m}$ and $\{\beta _{\ell m}\}_{\ell m}$ are 
mutually independent. Moreover, 
\begin{equation}
\mathbb{E}u_{\ell m}(t)=0,\quad \mathbb{E[}u_{\ell m}(t)\overline{u_{\ell ^{\prime
}m^{\prime }}}(s)]=\delta _{\ell }^{\ell ^{\prime }}\delta _{m}^{m^{\prime
}}U_{\ell }(t,s),  \label{Cov:ulm}
\end{equation}%
where 
\begin{equation}
U_{\ell }(t,s)=e^{-\ell (\ell +1)(t+s)}D_{\ell }+C_{\ell
}\int_{0}^{t}\int_{0}^{s}e^{-\ell (\ell +1)(t+s-\lambda -\xi )}|\xi -\lambda
|^{2H-2}d\xi d\lambda  \label{def:Ul}
\end{equation}%
in view of (\ref{fBm-inprod}). Here $U_{\ell }(t,s)=U_{\ell }(s,t)$ for any $s,t\in \mathbb{T}$. 

Now we prove the following lemma which implies that the series in (\ref{def:Solution}) 
converges uniformly both in the senses of $L^{2}(\Omega )$ 
for all $(t,x)\in \mathbb{T\times S}^{2}$ and $L^{2}(\Omega \times \mathbb{T})$ 
for all $x\in {\mathbb{S}}^{2}$.

\begin{lemma}
\label{Lem:convergence} 
Assume $1/2<H<1$, Conditions \textbf{(A.1) } and \textbf{(A.2)}. Then there exists a constant $K_{3,1}>0$,
depending on $\alpha ,\beta $ and $H$, such that for any $L$ large, 
\begin{equation}
\mathbb{E}\bigg\vert u(t,x)-\sum_{\ell =0}^{L}u_{\ell }(t,x)\bigg\vert^{2}
\leq K_{3,1}\left\{ t^{\beta /2-1}L^{-\beta }e^{-L(L+1)t}
+L^{-(\alpha-2+4H)}\right\}  
\label{Eq:conv-t}
\end{equation}%
for any fixed $(t,x)\in \mathbb{T\times S}^{2}$. Moreover, 
\begin{equation}
\mathbb{E}\bigg\Vert u(\cdot ,x)-\sum_{\ell =0}^{L}u_{\ell }(\cdot ,x)\bigg\Vert_{L^{2}\left(\mathbb{T}\right) }^{2}
\leq K_{3,1}\max \left\{ T^{\beta /2-1},T\right\} L^{-\min \left\{ \beta -2,\ \alpha -2+4H\right\} }
\label{Eg:conv-Int}
\end{equation}%
for every $x\in {\mathbb{S}}^{2}$.
\end{lemma}


\begin{proof}
Recall equation (\ref{coef-u(t)}) for $u_{\ell m}(t)$. We have 
\begin{equation*}
\mathbb{E}\bigg|u(t)-\sum_{\ell =0}^{L}u_{\ell }(t)\bigg|^{2} 
=\mathbb{E}\bigg|\sum_{\ell =L+1}^{\infty }u_{\ell }(t)\bigg|^{2} =I_{1}+I_{2} ,
\end{equation*}%
where 
\begin{equation*}
I_{1}=\mathbb{E} \bigg\vert \sum_{\ell =L+1}^{\infty }\sum_{m=-\ell }^{\ell}
e^{-\ell (\ell +1)t}u_{0,\ell m}Y_{\ell m} \bigg\vert ^{2}
\end{equation*}%
and 
\begin{equation*}
I_{2}=\mathbb{E}\bigg\vert \sum_{\ell =L+1}^{\infty }\sum_{m=-\ell }^{\ell } 
\sqrt{C_{\ell }}\int_{0}^{t}e^{-\ell (\ell +1)(t-s)}d\beta _{\ell m}(s)Y_{\ell m} \bigg\vert ^{2},
\end{equation*}%
in view of (\ref{Cov:ulm}). It is readily seen that under Conditions \textbf{(A.1)} and \textbf{(A.2)}, 
\begin{equation}
I_{1}=\sum_{\ell =L+1}^{\infty }e^{-\ell (\ell +1)t} \frac{2\ell +1}{4\pi }D_{\ell } 
\leq c_{3,1}t^{\beta /2-1}L^{-\beta }e^{-L(L+1)t}
\label{ineq:I1}
\end{equation}%
in view of the equality (\ref{SumYlm^2}). Here $c_{3,1}$ is some positive
constant depending on $\beta $. Hence, 
\begin{equation}
\int_{0}^{T}I_{1}dt\leq c_{3,2}T^{\beta /2-1}L^{-\beta -2}
\label{ineq:IntI1}
\end{equation}%
for some constant $c_{3,2}>0.$ For the second sum $I_{2}$, we have 
\begin{equation*}
I_{2}=\sum_{\ell =L+1}^{\infty }\sum_{m}C_{\ell }
\mathbb{E} \left\vert \int_{0}^{t}e^{-\ell (\ell +1)(t-s)}d\beta _{\ell m}(s) \right\vert^{2}
|Y_{\ell m}|^{2}
\end{equation*}%
in view of the properties (\ref{conj-beta}) and (\ref{indep-beta}) of $\beta_{\ell m}$ and the equality 
(\ref{SumYlm^2}). It is readily seen that the
process $\int_{0}^{t}e^{-\ell (\ell +1)(t-s)}d\beta _{\ell m}(s)$ is
Gaussian with mean zero. Now we focus on its variance, which is 
\begin{equation*}
\begin{split}
\sigma _{\ell }^{2}(t)&=: \mathbb{E}\left\vert \int_{0}^{t}e^{-\ell (\ell+1)(t-s)}d\beta _{\ell m}(s)\right\vert^{2}\\
&=\int_{0}^{t}\int_{0}^{t}e^{-\ell (\ell +1)(2t-s-r)}|s-r|^{2H-2}drds.
\end{split}
\end{equation*}%
Let 
\begin{equation}
g_{\ell }(s,\lambda ) =e^{-\ell (\ell +1)(s-\lambda )}1_{[0,s]}(\lambda ).
\label{def:g_l}
\end{equation}%
It is known that (see for instance \cite{BalTud08}, Lemma A.1.) 
\begin{equation}
\int_{0}^{T}\int_{0}^{T}g_{\ell }(t,\lambda )g_{\ell }(s,\xi ) |\xi -\lambda|^{2H-2}d\xi d\lambda 
=c_{H}\int_{\mathbb{R}}\widehat{g_{\ell }}(t,\tau ) \overline{\widehat{g_{\ell }}(s,\tau )}\frac{d\tau}{|\tau |^{2H-1}} ,
\label{Fourier transform}
\end{equation}%
where $c_{H}=\left[ 2^{2(1-H)}\pi ^{1/2}\right] ^{-1} \Gamma (H-1/2)/\Gamma(1-H)$ and 
\begin{equation}
\widehat{g_{\ell }}(t,\tau ) 
=\int_{0}^{t}e^{-i\tau \lambda } e^{-\ell (\ell +1)(t-\lambda )}d\lambda
 =\frac{e^{-it\tau }-e^{-\ell (\ell +1)t}}{\ell (\ell +1)-i\tau }.  \label{Fourier-g_l}
\end{equation}%
Moreover, 
\begin{equation}
\int_{\mathbb{R}}\frac{|\tau |^{-(2H-1)}} {[\ell (\ell +1)]^{2}+\tau ^{2}}d\tau 
=c_{3,3}[\ell (\ell +1)]^{-2H}  \label{int-tau}
\end{equation}%
with $c_{3,3}=\int_{\mathbb{R}}\frac{|\tau |^{-(2H-1)}}{1+\tau ^{2}}d\tau $. Thus, 
\begin{equation}
\begin{split}
I_{2}& =\sum_{\ell =L+1}^{\infty }\frac{2\ell +1}{4\pi } C_{\ell }\sigma_{\ell }^{2}(t) 
\leq \frac{c_{0}c_{3,3}}{2\pi }\sum_{\ell=L+1}^{\infty } (\ell +\frac{1}{2})^{1-\alpha -4H} \\
& \leq c_{3,4}L^{-(\alpha -2+4H)},
\end{split}
\label{ineq:I2}
\end{equation}%
with some positive constant $c_{3,4}$ depending on $\alpha $ and $H,$ which
leads to 
\begin{equation}
\int_{0}^{T}I_{2}dt\leq c_{3,4}L^{-(\alpha -2+4H)}T.  
\label{ineq:IntI2}
\end{equation}%
Combining inequalities (\ref{ineq:I1}), (\ref{ineq:IntI1}), (\ref{ineq:I2})
and (\ref{ineq:IntI2}), we obtain the approximations in Lemma \ref{Lem:convergence}.
\end{proof}

Based on the uniform convergence (\ref{Eq:conv-t}) in Lemma \ref{Lem:convergence} and  
(\ref{Cov:ulm}), we have for every pair of 
$(t,s)\in \mathbb{T}\times \mathbb{T}$, 
\begin{equation*}
\begin{split}
& {\mathbb{E}}\left[ u(t,x)u(s,y)\right] =\sum_{\ell =0}^{\infty }
U_{\ell }(t,s)\sum_{m=-\ell }^{\ell }Y_{\ell m}(x)\overline{Y_{\ell m}}(y) \\
& =\sum_{\ell =0}^{\infty }\frac{2\ell +1}{4\pi }U_{\ell }(t,s)P_{\ell }(\left\langle x,y\right\rangle )
=\sum_{\ell =0}^{\infty }\frac{2\ell +1}{4\pi }U_{\ell }(t,s)P_{\ell }(\left\langle gx,gy\right\rangle ) \\
& ={\mathbb{E}}\left[ u(t,gx)u(s,gy)\right] 
\end{split}%
\end{equation*}%
for all $g\in SO(3)$ and all $x,y\in \mathbb{S}^{2}$. In particular, for every $t \in \mathbb T$, the random field 
$u(t)=\{u(t, x), x \in \mathbb{S}^{2}\}$ is 2-weakly isotropic.

Now we are ready to give the following result, which is also part (a) of Theorem \ref{Th:Main}. More precise 
information on differentiability of $u(t, x)$ in variable $x$ can be found in Corollary \ref{Cor:derivative-x} below.

\begin{theorem}
\label{Th:Existence} Assume $1/2<H<1$ and Conditions \textbf{(A.1)} and \textbf{(A.2) } hold with $\alpha +4H>4$. 
Then the following properties hold:
\begin{enumerate}
\item[(i)] for every $t\in \mathbb{T}$,  we have $u(t,\cdot )\in C^{1}(\mathbb{S}^{2})$\ a.s.

\item[(ii)] if $\alpha +4H>6$, then for every $t\in \mathbb T$, $u(t,\cdot )\in C^{2}(\mathbb{S}^{2})$\ a.s. and, 
for every $x\in \mathbb{S}^{2}$, $u(\cdot ,x)\in \mathbb{H}^{1}(\mathbb{T})$\ a.s. 
\end{enumerate}
\end{theorem}

\begin{proof}
For part (i), we only need to prove that $\mathbb{E}|\nabla _{\mathbb{S}^{2}}u(t)|^{2}<\infty $. Similar to the 
argument for obtaining $\mathbb{E}|\nabla _{{\mathbb{S}}^{2}}W^{H}(t,x)|^{2}$ 
in the proof of Proposition \ref{Prop: W^H},  we have 
\begin{equation*}
\mathbb{E}|\nabla _{\mathbb{S}^{2}}u|^{2}
=\sum_{\ell \geq 0}\ell (\ell +1)\frac{2\ell +1}{4\pi }U_{\ell }(t,t)
\end{equation*}%
in view of Lemma \ref{Lem:convergence}. Now recall (\ref{def:Ul}) and (\ref{ineq:I2}), we have 
\begin{equation*}
\begin{split}
\mathbb{E}|\nabla _{\mathbb{S}^{2}}u|^{2} 
&=\sum_{\ell \geq 0}\ell (\ell +1)\frac{2\ell +1}{4\pi }e^{-\ell (\ell +1)(t+s)}D_{\ell } \\
&\qquad +\sum_{\ell \geq 0}\ell (\ell +1)\frac{2\ell +1}{4\pi }C_{\ell }
\mathbb{E}\bigg\vert \int_{0}^{t}e^{-\ell (\ell +1)(t-s)}d\beta _{\ell m}(s)\bigg \vert ^{2} \\
&\leq c_{3,5}\bigg[ t^{\beta /2-1}+\sum_{\ell =0}^{\infty }(\ell +\frac{1}{2})^{3-\alpha -4H}\bigg]  \\
&\leq c_{3,6},
\end{split}
\end{equation*}%
where the last inequality is true because $\alpha +4H>4$. The constants $c_{3,5}$ and $c_{3,6}$ are finite and
positive depending only on $T,$ $\alpha ,\ \beta $ and $H.$  

For part (ii), notice that by eq.(\ref{def:heat eq}) we have
\begin{equation*}
\int_{0}^{T}|D_{t}u|^{2}dt
\leq \int_{0}^{T}|\Delta _{\mathbb{S}^{2}}u|^{2}dt
+\int_{0}^{T}|D_{t}W^{H}|^{2}dt.
\end{equation*}%
Since 
\begin{equation*}
\Delta _{\mathbb{S}^{2}}u=-\sum_{\ell =0}^{\infty }\ell (\ell +1)u_{\ell }
\end{equation*}%
in view of Lemma \ref{Lem:convergence} and the fact that $\alpha +4H>6$, we have 
\begin{equation}\label{Lap-u}
\begin{split}
\mathbb{E}\int_{0}^{T}|\Delta _{\mathbb{S}^{2}}u|^{2}dt 
&=\sum_{\ell =0}^{\infty }\ell ^{2}(\ell +1)^{2}\frac{2\ell +1}{4\pi }
e^{-\ell (\ell +1)(t+s)}D_{\ell }    \\
&\quad +\sum_{\ell \geq 0}\ell ^{2}(\ell +1)^{2}\frac{2\ell +1}{4\pi }C_{\ell }
\mathbb{E}\bigg \vert \int_{0}^{t}e^{-\ell (\ell +1)(t-s)}d\beta _{\ell m}(s)\bigg\vert ^{2}    \\
&\leq c_{3,7}\bigg[ t^{\beta /2-1}+\sum_{\ell =0}^{\infty }(\ell +\frac{1}{2})^{5-\alpha -4H}\bigg]   \\
&\leq c_{3,8},  
\end{split}
\end{equation}
for  some positive constants $c_{3,7}$ and $c_{3,8}$ depending only on $T,$ $\alpha ,\beta $ and $H$. Thus, we
have  proved the first conclusion in part (ii) as well as the second one since 
\begin{equation*}
\mathbb{E}  \int_{0}^{T}|D_{t}u|^{2}dt   <\infty ,
\end{equation*}%
in view of (\ref{Lap-u}) and Proposition \ref{Prop: W^H}. Hence the proof is
completed. 
\end{proof}

\section{Some Technical Tools \label{Sec: TechnicalTools}}

In this section, we study the variogram and strong local nondeterminism of the solution (\ref{mild-form}) . 
These properties are the key for investigating the exact modulus of continuity for the solution $u(t,x)$.

\subsection{Estimation of the Variogram}


\begin{proposition}
\label{Prop:Variogram} 
Assume that $1/2<H<1$ and  Conditions \textbf{(A.1)} and \textbf{(A.2)} hold. Then the solution $u(t,x)$ 
defined in \eqref{mild-form} satisfies the following condition: 
There exist constants $K_{4,1}>0$ and $0<\delta <1,$ such that for any $|t-s|<\delta $, $\theta 
=d_{\mathbb{S}^{2}}(x,y)<\delta $, 
\begin{equation}\label{Eq:vaio}
\mathbb{E}|u\left( t,x\right) -u\left( s,y\right) |^{2}
\leq K_{4,1} \big[|t-s|^{2\eta }+ \rho^{2} _{\gamma }(\theta) \big],
\end{equation}%
where $\eta =H-\max\{(2-\alpha)/4,0\}$, $\gamma = \alpha/2-1+2H$, and the function
$\rho _{\gamma }:\mathbb{R}_{+}/\{0\}\rightarrow\mathbb{R}_{+}$ is defined as follows:
\begin{equation}
\rho _{\gamma }(s)=\left\{
\begin{array}{cc}
s^{\gamma }, & \hbox{ if } \gamma <1, \\
s\sqrt{|\ln s|}, & \hbox{ if } \gamma =1, \\
s, & \hbox{ if } \gamma >1.
\end{array}%
\right.
\label{def:rho}
\end{equation}
\end{proposition}

\begin{proof}
For any $( t,x ) ,( s,y ) \in \mathbb{T}\times {\mathbb{S}}^{2}$, we have
\begin{equation}\label{vario1}
\begin{split}
&\mathbb{E}|u\left( t,x\right) -u\left( s,y\right) |^{2}  \\   
&=\mathbb{E}\bigg\vert
\sum_{\ell =0}^{\infty }\sum_{m=-\ell }^{\ell }
u_{\ell m}(t)Y_{\ell m}(x)-u_{\ell m}(s)Y_{\ell m}(y)
\bigg\vert ^{2}                               \\
&=\sum_{\ell =1}^{\infty }\frac{2\ell +1}{4\pi }
\big[
U_{\ell }(t,t)+U_{\ell }(s,s)
-2U_{\ell }(t,s)P_{\ell }(\left\langle x,y\right\rangle )
\big]
\end{split}
\end{equation}
in view of (\ref{Cov:ulm}) and (\ref{SumYlm^2}). Recall (\ref{def:Ul}), we obtain
\begin{equation}
U_{\ell }(t,t)+U_{\ell }(s,s)
-2U_{\ell }(t,s)P_{\ell }(\left\langle x,y\right\rangle )
=D_{\ell }A_{\ell ,1}+C_{\ell }A_{\ell ,2},
\label{vario2}
\end{equation}%
where
\begin{equation*}
A_{\ell ,1}
=:e^{-2\ell (\ell +1)t}+e^{-2\ell (\ell +1)s}
-2e^{-\ell (\ell +1)(t+s)}P_{\ell }(\left\langle x,y\right\rangle ),
\end{equation*}%
and
\[
\begin{split}
A_{\ell ,2}
&=:\sum_{\mu =t,s}\int_{0}^{\mu }\int_{0}^{\mu }
e^{-\ell (\ell +1)(2\mu -\lambda -\xi )}
|\lambda -\xi |^{2H-2}d\xi d\lambda
 \\
&\qquad-2P_{\ell }(\left\langle x,y\right\rangle )
\int_{0}^{t}\int_{0}^{s}e^{-\ell (\ell +1)(t+s-\lambda -\xi )}
|\lambda -\xi|^{2H-2}d\xi d\lambda .
\end{split}%
\]
Hence  \eqref{Eq:vaio} follows from Lemmas \ref{Lem:A_l1} and \ref{Lem:A_l2}
below.
\end{proof}

\begin{lemma}
\label{Lem:A_l1} There exists a constant $K_{4,2}>1$ such that for any $0<t-s,\ \theta <\delta $, we have
\begin{equation*}
\sum_{\ell =0}^{\infty }\frac{2\ell +1}{4\pi }D_{\ell }A_{\ell ,1}
\leq K_{4,2}\left\{
(t-s)^{2\wedge (\beta /2-1)}
+\left[ 1+(t+s)^{\beta /2-2}\right]\theta ^{2}
\right\} .
\end{equation*}
\end{lemma}

\begin{proof}
Let $\left\langle x,y\right\rangle =\cos \theta $ where $\theta \in \left[0,\pi \right] $ is the geodesic distance 
between $x$ and $y$ on $\mathbb{S}^{2}$, then we have
\begin{equation} \label{def:A_l1}
\begin{split}
A_{\ell ,1}
&=\left[ e^{-\ell (\ell +1)t}-e^{-\ell (\ell +1)s}\right] ^{2}
+2e^{-\ell (\ell +1)(t+s)}\left[ 1-P_{\ell }(\cos \theta )\right]   \\
&=:\widetilde{A}_{\ell ,1}^{0}+\widetilde{A}_{\ell ,1}^{1}.
\end{split}
\end{equation}%
Notice that
\begin{eqnarray}
\label{ineq:A_l11}
\sum_{\ell =0}^{\infty }\frac{2\ell +1}{4\pi }
D_{\ell }\widetilde{A}_{\ell ,1}^{0}
&&\leq \int_{s}^{t}\int_{s}^{t}\left[
\int_{1}^{\infty }x^{5-\beta }e^{-x^{2}(w+v)}dx
\right] dwdv  \notag \\
&&\leq \int_{s}^{t}\int_{s}^{t}(w+v)^{\beta /2-3}\, dwdv\\
&&\leq \left\{
\begin{array}{ll}
(2t)^{\beta /2-3}(t-s)^{2}, & \hbox{if }\, \beta \geq 6, \\
\frac 2 {\beta - 4}(t-s)^{\beta /2-1}, & \hbox{if }\, 4<\beta <6,
\end{array}%
\right. \notag
\end{eqnarray}
In deriving the last inequality, we have used the elementary fact that, 
if  $4<\beta <6$, then $0 < \frac \beta 2 - 2 < 1$ and, consequently,  
$b^{\beta/2-2 }-a^{\beta/ 2 - 2}\leq (b-a)^{ \beta/2 - 2}$ for all
$0<a<b<\infty$. 


Now let us focus on $\widetilde{A}_{\ell ,1}^{1}$. Recall the following
Hilb's asymptotics (see \cite{szego}, page 195, Theorem 8.21.6 or
\cite{GradRyzh} 8.722): there exists a constant $c_{4,1}>0,$ such that,
uniformly for all $\theta \in (0,\pi ),$
\begin{equation*}
P_{\ell }(\cos \theta )
=\left\{ \frac{\theta }{\sin \theta }\right\}^{1/2}
J_{0}\Big((\ell +\frac{1}{2})\theta \Big)
+\delta _{\ell }(\theta ),
\end{equation*}%
where
\begin{equation*}
\delta _{\ell }(\theta )
<< \left\{
\begin{array}{ll}
\theta ^{2}O(1) & \hbox{ for } 0<\theta <c_{4,1}\ell ^{-1},  \\
\theta ^{1/2}O(\ell ^{-3/2}) & \hbox{ for } \theta >c_{4,1}\ell ^{-1}
\end{array}%
\right. 
\end{equation*}%
and $J_{0}$ is the\emph{\ Bessel function} defined as
\begin{equation*}
J_{0}(x)=\sum_{k=0}^{\infty}\frac{(-1)^{k}}{(k!)^{2}}\Big(\frac{x}{2}\Big)^{2k}
\end{equation*}%
(c.f. \cite{GradRyzh}, 8.402), which yields that
\begin{equation*}
\lim_{u\rightarrow 0}
\frac{1-J_{0}(c_{4,1}u) }{c_{4,1}^{2}u^{2}}
=\frac{1}{2}.
\end{equation*}%
Thus, by using the fact that
\begin{equation*}
\frac{\theta }{\sin \theta }-1
=\frac{\theta ^{2}}{6}+O(\theta ^{3})
\hbox{, \ as } \ \theta \rightarrow 0,
\end{equation*}%
we obtain that for any positive integer $L<c_{4,1}\theta ^{-1}$,
\begin{equation}
\begin{split}
&\sum_{\ell =1}^{L}\frac{2\ell +1}{4\pi }D_{\ell }
e^{-\ell (\ell +1)(t+s)}\big\{ 1-P_{\ell }(\cos \theta )\big\} \\
&\leq D_{0}\sum_{\ell =1}^{L}(\ell +\frac{1}{2})^{1-\beta }
e^{-\ell (\ell +1)(t+s)}
\bigg( \frac{\ell ^{2}}{c_{4,1}^{2}}-\frac{1}{6}\bigg) \theta ^{2}\\
&\leq c_{4,2}e^{(t+s)/4}(t+s)^{\beta /2-2}
\cdot C(t+s,\theta )\theta ^{2},
\end{split}
\end{equation}
where $c_{4,2}$ is a positive constant depending on $D_{0},$ $c_{4,1}$ and $\beta $, and
\[
\begin{split}
C(t+s,\theta )
&=\left[ \Gamma (2-\beta /2,t+s) -\Gamma (2-\beta/2,c_{A}^{2}(t+s)\theta ^{-2}) \right] \\
&\leq C(1+(t+s)^{2-\beta /2})e^{-(t+s)},
\end{split}%
\]
with $\Gamma (a,x)$ the incomplete Gamma function defined as
\begin{equation*}
\Gamma (a,x)=\int_{x}^{\infty }e^{-u}u^{a-1}du
\end{equation*}%
(c.f. \cite{GradRyzh} 3.381 and 8.350). On the other hand, recall that $1-P_{\ell }(\cos \theta )\leq 2$ 
uniformly for all $\theta $, whence we have, for any $U>\frac{c_{4,1}}{\theta }$,
\[
\begin{split}
&\sum_{\ell =U}^{\infty }\frac{2\ell +1}{4\pi }D_{\ell }
e^{-\ell (\ell +1)(t+s)} \left\{ 1-P_{\ell }(\cos \theta )\right\} \\
&\leq  D_{0}\int_{U}^{+\infty }
(\ell +\frac{1}{2})^{1-\beta }e^{-\ell (\ell +1)(t+s)}d\ell \\
&\leq D_{0}e^{(t+s)/4}(t+s)^{\beta /2-1}
\int_{c_{4,1}^{2}(t+s)/\theta^{2}}^{\infty }e^{-u}u^{-\beta /2}du \\
&\leq D_{0}e^{(t+s)/4}(t+s)^{\beta /2-2}
\Gamma (2-\beta/2,c_{A}^{2}(t+s)\theta ^{-2})\cdot \theta ^{2}.
\end{split}%
\]
Thus,
\begin{equation}
\sum_{\ell =0}^{\infty }\frac{2\ell +1}{4\pi }
D_{\ell }e^{-\ell (\ell +1)(t+s)}
\left\{ 1-P_{\ell }(\cos \theta )\right\}
\leq c_{4,3}\left[1+(t+s)^{\beta /2-2}\right] \theta ^{2}  \label{ineq:A_l12}
\end{equation}%
for some constant $c_{4,3}>0$ depending on $D_{0}$. Lemma (\ref{Lem:A_l1})
is then obtained by (\ref{def:A_l1}), (\ref{ineq:A_l11}) and (\ref{ineq:A_l12}),
\end{proof}

Now let us focus on $A_{\ell ,2}$.

\begin{lemma}
\label{Lem:A_l2} There exists a constant $K_{4,3}>1$ such that for any $0<t-s,\ \theta <\delta $, 
we have
\begin{equation*}
\sum_{\ell =0}^{\infty }\frac{2\ell +1}{4\pi }C_{\ell }A_{\ell ,2}
\leq K_{4,3}\{(t-s)^{2\eta}+\rho^{2} _{\gamma }(\theta)\},
\end{equation*}
with $\eta$ and $\rho_{\gamma }$ defined in Proposition \ref{Prop:Variogram}.
\end{lemma}

\begin{proof}
Recall (\ref{def:g_l}), we decompose $A_{\ell ,2}$ into the following form
\begin{equation} \label{def:A_L2}
\begin{split}
A_{\ell ,2}
&= \int_{0}^{T}\int_{0}^{T}\left[ g_{\ell }(t,\lambda )-g_{\ell
}(s,\xi )\right] ^{2}
|\xi -\lambda |^{2H-2}d\xi d\lambda  \\
&\qquad +2\left[ 1-P_{\ell }(\left\langle x,y\right\rangle )\right]
\int_{0}^{t}\int_{0}^{s}e^{-\ell (\ell +1)(t+s-\lambda -\xi )}
|\lambda -\xi|^{2H-2}d\xi d\lambda   \\
&=:\widetilde{A}_{\ell ,2}^{0}+2\widetilde{A}_{\ell ,2}^{1}
\end{split}
\end{equation}%
and we divide the proof into two steps.

Step1: Approximation for $\widetilde{A}_{\ell ,2}^{0}$. Recall (\ref{Fourier transform}) and (\ref{Fourier-g_l}),
\[
\begin{split}
\widetilde{A}_{\ell ,2}^{0}
&=c_{H}\int_{\mathbb{R}}\left\vert
\widehat{g_{\ell }}(t,\tau )-\widehat{g_{\ell }}(s,\tau )\right
\vert^{2}  |\tau|^{-(2H-1)}d\tau \\
&\leq  c_{H}\int_{\mathbb{R}}
\frac {\left\vert e^{-\ell (\ell +1)s}-e^{-\ell (\ell +1)t}\right\vert ^{2}}
{\left[ \ell (\ell +1)\right] ^{2}+\tau ^{2}}
|\tau |^{-(2H-1)}d\tau \\
&\qquad +c_{H}\int_{\mathbb{R}}
\frac {\left\vert e^{-it\tau}-e^{-is\tau }\right\vert ^{2}}
{\left[ \ell (\ell +1)\right] ^{2}+\tau ^{2}}
|\tau |^{-(2H-1)}d\tau \\
&:=c_{H}\left\{ \widetilde{_{1}A}_{\ell,2}^{0}+\widetilde{_{2}A}_{\ell ,2}^{0}
\right\}.
\end{split}
\]
Recall (\ref{int-tau}) and use the fact
\begin{equation*}
\left\vert e^{-\ell (\ell +1)s}-e^{-\ell (\ell +1)t}\right\vert ^{2}
=\int_{s}^{t}\int_{s}^{t}e^{-\ell (\ell +1)(w+v)}dwdv,
\end{equation*}%
we have
\begin{equation}\label{eq2}
\begin{split}
\sum_{\ell =1}^{\infty }\frac{2\ell +1}{4\pi}
C_{\ell }\cdot \widetilde{_{1}A}_{\ell ,2}^{0}
&\leq c_{3,3}\sum_{\ell =1}^{\infty }\frac{2\ell +1}{4\pi }
C_{\ell }\int_{s}^{t}\int_{s}^{t}e^{-\ell (\ell +1)(w+v)}dwdv
\left[\ell(\ell +1)\right]^{2-2H}   \\
&\leq c_{4,4}  \int_{s}^{t}\int_{s}^{t}
(w+v)^{\alpha/2-3+2H}dwdv    \\
&\leq  c_{4,5}\left\vert t-s\right\vert ^{\min{\{\alpha/2-1+2H,2\}}}.
\end{split}
\end{equation}%
In the meantime,
\begin{equation}
\begin{split}
\sum_{\ell =1}^{\infty }\frac{2\ell +1}{4\pi }C_{\ell }
\cdot \widetilde{_{2}A}_{\ell ,2}^{0}
&\leq \sum_{\ell =1}^{\infty }
\frac{(2\ell +1)C_{\ell }} {4\pi \left[ \ell (\ell +1)\right] ^{2}}
\int_{\mathbb{R}} \left\vert e^{-it\tau }-e^{-is\tau }\right\vert ^{2}
|\tau |^{-(2H-1)}d\tau   \\
&\leq c_{4,6}c_{H}\int_{s}^{t}\int_{s}^{t}
|\lambda -\xi |^{2H-2}d\xi d\lambda
=c_{4,7}|t-s|^{2H}.
\label{eq3}
\end{split}%
\end{equation}
Here and above, $c_{4,4}, \ldots, c_{4,7}$ are  positive
constants depending on $c_{0},$ $\alpha $ and $H.$ Therefore, combining
inequalities (\ref{eq2}) and (\ref{eq3}), we have that for $|t-s|$ small
enough,
\begin{equation}
\sum_{\ell =1}^{\infty }\frac{2\ell +1}{4\pi }
C_{\ell }\widetilde{A}_{\ell,2}^{0}
\leq 2c_{4,6}|t-s|^{2H-\max{\{1-\alpha/2,0\}}}.
\label{ineq:A_l21}
\end{equation}%

Step2: Approximation for $\widetilde{A}_{\ell ,2}^{1}$. Recall (\ref{Fourier transform}), 
(\ref{Fourier-g_l}) and (\ref{int-tau}), we have%
\begin{equation*}
\begin{split}
&\left\vert \int_{0}^{t}\int_{0}^{s}
e^{-\ell (\ell +1)(t+s-\lambda -\xi )}
|\lambda -\xi|^{2H-2}d\xi d\lambda
\right\vert \\
&\leq  c_{H}\int_{\mathbb{R}}\widehat{g_{\ell }}(t,\tau)
\overline{\widehat{g_{\ell }}(s,\tau )}
|\tau |^{-(2H-1)}d\tau
=c_{H}c_{3,3}\left[ \ell (\ell +1)\right] ^{-2H}.
\end{split}
\end{equation*}%
Therefore, by \cite{LanMarXiao}\ Lemma 10, we have
\begin{equation}\label{ineq:A_l22}
\begin{split}
\sum_{\ell =0}^{\infty }\frac{2\ell +1}{4\pi }
C_{\ell }\widetilde{A}_{\ell,2}^{1}
&\leq c_{0}c_{H}c_{4,4}\sum_{\ell =0}^{\infty }
(\ell +\frac{1}{2})^{1-\alpha -4H}
\left[ 1-P_{\ell }(\left\langle x,y\right\rangle )\right]\\
&\leq c_{4,8}\rho^{2} _{\gamma }(\theta)
\end{split}
\end{equation}%
for some positive constant $c_{4,8}$ depending on $c_{0},$ $\alpha $ and $H$. Hence, 
combining inequalities (\ref{def:A_L2}) and (\ref{ineq:A_l21}) together with (\ref{ineq:A_l22}), 
we obtain Lemma (\ref{Lem:A_l2}).
\end{proof}

\subsection{Strong Local Nondeterminism}

In this section we prove the properties of strong local nondeterminism of the solution $\{u(t, x): t \in \mathbb T,\, x \in \mathbb{S}^2\}$
in   time variable $t\in \mathbb{T}$ and  spatial variable $x\in \mathbb{S}^{2}$, respectively. 

First, let $x\in \mathbb{S}^{2}$ be fixed and we consider  the Gaussian process $\{ u(t, x), t \in \mathbb T\}$.
Without loss of generality, we write $u\left( t\right) =u\left( t,x\right) $
for brevity. 

\begin{proposition}
\label{Prop:SLND-t} Assume $1/2<H<1$ and  Conditions \textbf{(A.1)} and \textbf{(A.2)}.  Then 
there exist  constants $K_{4,4}>0$ and $0<\varepsilon <\delta $ (which do not depend on $x\in \mathbb{S}^{2}$) such
that  for all $t\in (0,T]$ and $r\in (0,\varepsilon )$,
\begin{equation} \label{Ineq:SLND}
\mathrm{Var}\left( u\left( t\right) |u\left( s\right) :s\in \mathbb{T},t-s\geq r \right) 
\geq K_{4,4}r^{2\eta},  
\end{equation}
where $\eta=H-\max{\{(2-\alpha)/4,0\}}$.
\end{proposition}

\begin{proof} The proof is inspired by the proof of \cite[Theorem 2.1]{Xiao(SLND)}, but with a modification.
It is sufficient to prove that, there exists some positive constant $c_{4,9} $ such that
\begin{equation}\label{Eq:Slnd2}
V_{t}=:\mathbb{E}\bigg\vert u\left( t\right) -\sum_{j=1}^{n}a_{j}u\left( t_{j}\right) \bigg\vert ^{2}
\geq c_{4,9}r^{2H}
\end{equation}
for all integers $n\geq 1$ and all $t_{1},...,t_{n}\in
\mathbb{T}$ satisfying $|t-t_{j}|\geq r$.

Similar to (\ref{vario1}), we have
\[
\begin{split}
V_{t} &=\mathbb{E}\bigg\vert
\sum_{\ell =0}^{\infty }\sum_{m=-\ell }^{\ell }
u_{\ell m}(t)Y_{\ell m}(x)
-\sum_{j=1}^{n}a_{j}u_{\ell m}(t_{j})Y_{\ell m}(x) \bigg\vert ^{2} \\
&= \sum_{\ell =0}^{\infty }\frac{2\ell +1}{4\pi }
\bigg[ U_{\ell }(t,t)+\sum_{i,j=1}^{n}a_{i}a_{j}U_{\ell }(t_{i},t_{j})
-\sum_{j=1}^{n}a_{j}U_{\ell }(t,t_{j})) \bigg].
\end{split}
\]
Recall (\ref{def:Ul}), we have
\[
\begin{split}
V_{t} &=\sum_{i,j=0}^{n}a_{i}a_{j}
\sum_{\ell =0}^{\infty }\frac{2\ell +1}{4\pi }C_{\ell}
\int_{0}^{t_{i}}\int_{0}^{t_{j}}
e^{-\ell (\ell +1)(t_{i}+t_{j}-\xi -\lambda )}
|\xi -\lambda |^{2H-2}d\xi d\lambda
\\
&\qquad +\sum_{\ell =0}^{\infty }\frac{2\ell +1}{4\pi }D_{\ell }
\bigg\vert e^{-\ell (\ell +1)t}-\sum_{j=1}^{n}a_{j}e^{-\ell (\ell +1)t_{j}} \bigg\vert ^{2} \\
&\geq \sum_{i,j=0}^{n}a_{i}a_{j}
\sum_{\ell =0}^{\infty }\frac{2\ell +1}{4\pi }C_{\ell }
\int_{0}^{t_{i}}\int_{0}^{t_{j}}
e^{-\ell (\ell +1)(t_{i}+t_{j}-\xi -\lambda )}
|\xi -\lambda |^{2H-2}d\xi d\lambda
\\
&:=M(\mathbf{t}),
\end{split}%
\]
where the coefficient $a_{0}=-1$.

Now we construct a bump function $\delta _{t,r}(\cdot )\in S(\mathbb{R})$
(the Schwartz space on $\mathbb{R}$) for any $t>0$ such that its Fourier
transform $\widehat{\delta _{t,r}}$ vanishes outside the open interval $(t-r,t+r)$. Let
\begin{equation*}
\widehat{\delta _{t.r}}\left( \lambda \right)
=\left\{
\begin{array}{cc}
\exp \left\{ -\frac{r^{2}}{r^{2}-(\lambda -t)^{2}}\right\} ,
& \hbox{ if } |\lambda -t|<r,
\\
0, & \hbox{ otherwise}.%
\end{array}%
\right. 
\end{equation*}%
Then $supp$ $\widehat{\delta _{t,r}}\subseteq (t-r,t+r)$. Note that
\[
\begin{split}
\delta _{t,r}(\tau ) &=\int_{-\infty }^{\infty }e^{i\tau \lambda }
\widehat{\delta _{t,r}}\left( \lambda \right) d\lambda
=\int_{t-r}^{t+r}\exp \Big(
i\tau \lambda -\frac{r^{2}}{r^{2}-(\lambda -t)^{2}}\Big) d\lambda \\
&=r\, e^{i\tau t}\int_{-1}^{1}
\exp \left\{ i\tau ry-\frac{1}{1-y^{2}}\right\} dy.
\end{split}%
\]
Simple calculation yields that
\begin{equation*}
|\delta _{t,r}\left( \tau \right) |\leq 2r, \
\hbox{ for any }\tau \in \mathbb{R}
\end{equation*}%
and
\begin{equation*}
|\delta _{t,r}\left( \tau \right) |
=\bigg\vert \int_{-\infty }^{\infty}
e^{i\tau \lambda }\widehat{\delta _{t,r}}\left( \lambda \right) d\lambda
\bigg\vert 
\approx r\tau ^{-3/4}e^{-\sqrt{\tau }/2}, \ \
\hbox{ as }\ \tau \rightarrow \infty .
\end{equation*}%
See for instance \cite{Saddle-point} as well. Hence $\delta _{t,r}(\tau )\in S(\mathbb{R})$; that is, there 
exists $\delta _{t,r}\in S(\mathbb{R})$, such that $\widehat{\delta _{t,r}}\left( \tau \right) =
\int_{\mathbb{R}}\delta_{t,r}(s)e^{-i\tau s}ds$.

We are ready to prove \eqref{Eq:Slnd2}. Recall (\ref{Fourier
transform}) and (\ref{Fourier-g_l}), we can rewrite  $M(\mathbf{t})$   as
\begin{equation}
M(\mathbf{t})
=c_{H}\sum_{\ell =0}^{\infty }\frac{2\ell +1}{4\pi }C_{\ell }
\int_{-\infty }^{\infty }
\bigg\vert \widehat{g_{\ell }}\left( t,\tau \right) -\sum_{j=1}^{n}a_{j}\widehat{g_{\ell }}\left( t_{j},\tau \right)
\bigg\vert ^{2} \frac{d\tau}{|\tau |^{2H-1}}.
\label{def:M(t)}
\end{equation}%

We distinguish the two cases $\alpha >2$ and $0< \alpha \le 2$.
 
\noindent{(i).} If $\alpha >2,$ let
\begin{equation*}
G_{s}\left( \lambda \right)
=:\sum_{\ell =0}^{\infty }\frac{2\ell +1}{4\pi }
C_{\ell }g_{\ell }(s,\lambda ),\quad 
s\in \mathbb{T},\, \lambda \in \mathbb{R},
\end{equation*}%
then 
\[
\begin{split}
Q(s,t) &=:\int_{-\infty }^{\infty }
\widehat{G_{s}}\left( \tau \right)\delta _{t,r}(\tau )d\tau \\
&= \int_{0}^{s}G_{s}( \lambda )
\left[
\int_{-\infty }^{\infty}\delta _{t,r}(\tau )e^{-i\tau \lambda }d\tau
\right] d\lambda
=\int_{0}^{s}G_{s}( \lambda )
\widehat{\delta _{t,r}}(\lambda ) d\lambda.
\end{split}%
\]
It is readily seen that $supp\ Q\subseteq \left\{ (t,s):s>t-r\right\} $.
Hence $Q(t_{j},t)=0$ for any $0<t_{j}<t-r,\ j=1,...,n$. Moreover, let $0<r<\varepsilon <1$, then
\[
\begin{split}
Q(t,t) &=\int_{0}^{t}G_{t}\left( \lambda \right)
\widehat{\delta _{t,r}}( \lambda ) d\lambda
=\int_{t-r}^{t}G_{t}( \lambda )
\widehat{\delta _{t,r}}\left( \lambda \right) d\lambda \\
&\approx \int_{t-r}^{t}
\bigg[ \sum_{\ell =0}^{\infty }\frac{2\ell +1}{4\pi }
(\ell +\frac{1}{2})^{-\alpha }(t-\lambda )^{(1-\alpha )/2}
e^{-\ell (\ell +1)(t-\lambda )}\sqrt{t-\lambda }\bigg] \\
& \qquad \cdot \frac{e^{(t-\lambda )/4}}{(t-\lambda )^{1-\alpha /2}}
\widehat{\delta_{t,r}}\left(\lambda \right) d\lambda \\
&\approx \int_{t-r}^{t}
\big(1+(t-\lambda )^{\alpha /2-1}\big)\widehat{\delta _{t}}( \lambda ) d\lambda \\
&=r\int_{0}^{1}(1+(ry)^{\alpha /2-1})
\exp \Big( -\frac{1}{1-y^{2}}\Big) dy \geq r,
 \end{split}%
\]
which leads to 
\begin{equation*}
\bigg|Q(t,t)-\sum_{j=1}^{n}a_{j}Q(t_{j},t) \bigg|\geq r
\end{equation*}%
for $\max_{j=1,...,n}t_{j}\leq t-r$.
Meanwhile, recall the representation (\ref{def:M(t)}) of $M(\mathbf{t})$,
then by the Cauchy-Schwartz inequality, we have
\[
\begin{split} 
&\bigg| Q(t,t)-\sum_{j=1}^{n}a_{j}Q(t_{j},t) \bigg|^{2} \\
&=\bigg \vert \sum_{\ell =0}^{\infty }\sum_{m}
\frac{2\ell +1}{4\pi}C_{\ell}
\int_{-\infty }^{\infty }
\bigg( \widehat{g_{\ell }}\left( t,\tau \right)
-\sum_{j=1}^{n}a_{j}\widehat{g_{\ell }}\left( t_{j},\tau \right) \bigg)
\delta _{t}(\tau )d\tau
\bigg \vert ^{2} \\
&\leq 
\left\{ \sum_{\ell =0}^{\infty }\frac{2\ell +1}{4\pi }C_{\ell }\right\}
\sum_{\ell =0}^{\infty }\sum_{m}\frac{2\ell +1}{4\pi }C_{\ell }
\bigg\vert
\int_{-\infty }^{\infty }
\bigg( \widehat{g_{\ell }}\left( t,\tau
\right) -\sum_{j=1}^{n}a_{j}\widehat{g_{\ell }}\left( _{t_{j}},\tau \right)
\bigg) \delta _{t}(\tau )d\tau \bigg\vert ^{2} \\
&\leq  c_{4,10}\sum_{\ell =0}^{\infty }\frac{2\ell +1}{4\pi }C_{\ell}
\int_{-\infty }^{\infty }\bigg \vert \widehat{g_{\ell }}( t,\tau)
-\sum_{j=1}^{n}a_{j}\widehat{g_{\ell }}\left( _{t_{j}},\tau \right)
\bigg\vert ^{2}
|\tau |^{-(2H-1)}d\tau \\
& \qquad \cdot\left\{ \int_{-\infty }^{\infty}
|\delta _{t}(\tau )|^{2}|\tau |^{2H-1}d\tau
\right\} \\
&= c_{4,10}M(\mathbf{t})\int_{-\infty }^{\infty }
|\delta _{t}(\tau)|^{2}|\tau |^{2H-1}d\tau .
 \end{split}%
\] 
for some positive constant $c_{4,10}$ depending on $\alpha $. Now recall
formula (\ref{Fourier transform}) again, we have
\begin{eqnarray}
&&c_{H}\int_{-\infty }^{\infty }
|\delta _{t}(\tau )|^{2}|\tau |^{2H-1}d\tau
=c_{H}\int_{-\infty }^{\infty }|\delta _{t}(\tau )|^{2}
|\tau|^{1-2(1-H)}d\tau                         \notag \\
&=&\int_{t-r}^{t+r}\int_{t-r}^{t+r}
\widehat{\delta _{t}}(\lambda )\overline{\widehat{\delta _{t}}(\xi )}
|\xi -\lambda |^{2(1-H)-2}d\xi d\lambda            \notag \\
&=&r^{2-2H}\int_{-1}^{1}\int_{-1}^{1}\exp \left\{
 -\frac{1}{1-\xi ^{\prime 2}}-\frac{1}{1-\lambda ^{\prime 2}}
 \right\}
 |\xi ^{\prime }-\lambda ^{\prime}|^{-2H}
 d\xi ^{\prime }d\lambda^{\prime }            \notag  \\
&\leq &c_{4,13}r^{2(1-H)}
\label{int_delta}
\end{eqnarray}%
for some positive constant $c_{4,11}$ depending on $H$. Hence,
\begin{equation*}
M(\mathbf{t})\geq \frac{r^{2}}{c_{4,10}c_{4,11}r^{2-2H}}
=\frac{r^{2H}}{c_{4,10}c_{4,11}}.
\end{equation*}%

\noindent{(ii).} If $0<\alpha \le 2,$ let
\begin{equation*}
G_{s}\left( \lambda \right)
=:\sum_{\ell =0}^{\infty }(\ln \ell )^{-1}
\sqrt{C_{\ell }}g_{\ell }(s,\lambda ),\  \ \ s\in \mathbb{T},\, \lambda \in \mathbb{R}.
\end{equation*}%
Consider the function
\[
\begin{split}
Q(s,t) &=:\int_{-\infty }^{\infty }\widehat{G_{s}}\left( \tau \right)
\delta _{t,r}(\tau )d\tau  \\
&=\int_{0}^{s}G_{s}(\lambda )\left[ \int_{-\infty }^{\infty }
\delta_{t,r}(\tau )e^{-i\tau \lambda }d\tau \right] d\lambda
=\int_{0}^{s}G_{s}(\lambda )\widehat{\delta _{t,r}}(\lambda )d\lambda.
\end{split}%
\]
Then   $supp\ Q\subseteq \left\{ (t,s):s>t-r\right\}$, which implies
 $Q(t_{j},t)=0$ for any $0<t_{j}<t-r,\ j=1,...,n$.
Moreover, let $0<r<\varepsilon <1$, then
\[
\begin{split}
Q(t,t) &=\int_{0}^{t}G_{t}( \lambda )
\widehat{\delta _{t,r}}(\lambda )d\lambda
=\int_{t-r}^{t}G_{t}(\lambda )
\widehat{\delta _{t,r}}\left( \lambda \right) d\lambda  \\
&\approx \int_{t-r}^{t}\left[ \sum_{\ell =0}^{\infty }
(\ln \ell )^{-1}(\ell +\frac{1}{2})^{-\alpha /2}
(t-\lambda )^{-\alpha /4}
e^{-\ell (\ell +1)(t-\lambda )}\sqrt{t-\lambda }\right]  \\
&\qquad \cdot \frac{e^{(t-\lambda )/4}}{(t-\lambda )^{(2-\alpha )/4}}
\widehat{\delta _{t,r}}\left( \lambda \right) d\lambda  \\
&\geq \frac{1}{\ln 2}\int_{t-r}^{t}
(1+\frac{1}{(t-\lambda )^{(2-\alpha )/4}})
\widehat{\delta _{t}}(\lambda )d\lambda  \\
&=r\int_{0}^{1}(1+\frac{1}{(ry)^{(2-\alpha )/4}}
\exp \left\{ -\frac{1}{1-y^{2}}\right\} dy\geq r^{1-\max \{(2-\alpha )/4,0\}}.
\end{split}%
\]
It follows that 
\begin{equation*}
\bigg|Q(t,t)-\sum_{j=1}^{n}a_{j}Q(t_{j},t) \bigg|^{2}\geq r^{2-\max \{(2-\alpha )/2,0\}}
\end{equation*}%
for $\max_{j=1,...,n}t_{j}\leq t-r$. 
Meanwhile, recall the representation (\ref{def:M(t)}) of $M(\mathbf{t})$,
then by the Cauchy-Schwartz inequality, we have
\begin{eqnarray*}
&&|Q(t,t)-\sum_{j=1}^{n}a_{j}Q(t_{j},t)|^{2} \\
&=&\left\vert \sum_{\ell =0}^{\infty }\sum_{m}
\frac{2\ell +1}{4\pi}C_{\ell}
\int_{-\infty }^{\infty }
\left[ \widehat{g_{\ell }}\left( t,\tau \right)
-\sum_{j=1}^{n}a_{j}\widehat{g_{\ell }}\left( t_{j},\tau \right) \right]
\delta _{t}(\tau )d\tau
\right\vert ^{2} \\
&\leq &
\left\{ \sum_{\ell =0}^{\infty }\frac{2\ell +1}{4\pi }C_{\ell }\right\}
\sum_{\ell =0}^{\infty }\sum_{m}\frac{2\ell +1}{4\pi }C_{\ell }
\left\vert
\int_{-\infty }^{\infty }
\left[ \widehat{g_{\ell }}\left( t,\tau
\right) -\sum_{j=1}^{n}a_{j}\widehat{g_{\ell }}\left( _{t_{j}},\tau \right)
\right]
\delta _{t}(\tau )d\tau
\right\vert ^{2} \\
&\leq &c_{4,12}\sum_{\ell =0}^{\infty }\frac{2\ell +1}{4\pi }C_{\ell}
\int_{-\infty }^{\infty }\left\vert \widehat{g_{\ell }}( t,\tau)
-\sum_{j=1}^{n}a_{j}\widehat{g_{\ell }}\left( _{t_{j}},\tau \right)
\right\vert ^{2}
|\tau |^{-(2H-1)}d\tau \\
&&\cdot\left\{ \int_{-\infty }^{\infty}
|\delta _{t}(\tau )|^{2}|\tau |^{2H-1}d\tau
\right\} \\
&=&c_{4,12}M(\mathbf{t})\int_{-\infty }^{\infty }
|\delta _{t}(\tau)|^{2}|\tau |^{2H-1}d\tau .
\end{eqnarray*}%
for some universal constant $c_{4,12}>0$. The last inequality is obtained similar to the argument in (i) and hence by (\ref{int_delta}) we
have
\begin{equation*}
M(\mathbf{t})\geq \frac{r^{2-\max \{(2-\alpha )/2,0\}}}
{c_{4,10}c_{4,11}r^{2-2H}}
=\frac{r^{2\eta }}{c_{4,10}c_{4,11}}.
\end{equation*}%
Thus, the proof is completed.
\end{proof}

The next proposition is concerned with the strong local nondeterminism of $u\left( t,x\right) $ 
in space variable $x \in \mathbb{S}^2$, when $t\in \mathbb{T}$ fixed. Again, with a slight abuse
of notation, we write $u\left( x\right) =u\left( t,x\right) $ for brevity.

\begin{proposition}
\label{Prop:SLND-x} Assume that $1/2<H<1$,  Condition \textbf{(A.1)}  with $0 < \alpha < 2$,
and  $u_{0}\equiv 0$.  If $\gamma =\alpha /2-1+2H\in (0,1)$, then, for every $t\in (0,T]$ fixed, there exist 
constants $K_{4,5}>0$, depending on $t$ and $0<\varepsilon <\delta$, such
that we have for all $x_{0},x_{1},...,x_{n}\in \mathbb{S}^{2}$ with
$\min \left\{ d_{\mathbb{S}^{2}}(x_{0},x_{j}),j=1,...,n\right\} =r\in (0,\varepsilon )$,
\begin{equation} \label{Eq:Slnd2}
\mathrm{Var}\big( u\left(t, x_{0}\right) |u\left(t, x_{1}\right) ,...,u\left(t,x_{n}\right)
\big) \geq K_{4,5}r^{2\gamma }.
\end{equation}
\end{proposition}

\begin{proof}
Since the Gaussian random field $u(t)=\{u(t,x), x \in \mathbb{S}^2\}$ is 2-weakly isotropic, the results in  \cite{LanMarXiao}
is applicable. Hence, in order to prove \eqref{Eq:Slnd2}, we only need to derive the asymptotic property 
of the angular power spectrum of $\{u(x), x \in \mathbb{S}^2\}$. Recall (\ref{Cov:ulm}), under the condition of $u_{0}\equiv 0$, the angular power spectrum of $u(t,x)$  
\begin{equation*}
\mathbb{E}|u_{\ell m}(t)|^{2}=C_{\ell }\int_{0}^{T}\int_{0}^{T}g_{\ell }(t,\lambda )g_{\ell }(t,\xi )
|\xi-\lambda |^{2H-2}d\xi d\lambda
\end{equation*}
where $g_{\ell }$ is the function defined in (\ref{def:g_l}).
Recall formulae (\ref{Fourier transform}), (\ref{Fourier-g_l}) and (\ref{int-tau}), we have
\begin{equation*}
\mathbb{E} |u_{\ell m}(t)|^{2}=c_{H}C_{\ell }\int_{\mathbb{R}}
\frac{|e^{-it\tau }-e^{-\ell (\ell +1)t}|^{2}}{[\ell (\ell +1)]^{2}+\tau ^{2}}
|\tau |^{-(2H-1)}d\tau     
\end{equation*}
and there exists a constant $c_{4,13}>1$ such that 
\begin{equation}\label{APS-u}
c^{-1}_{4,13}(\ell +\frac{1}{2})^{-(\alpha +4H)} 
\leq \mathbb{E}|u_{\ell m}(t)|^{2}
\leq c_{4,13}(\ell +\frac{1}{2})^{-(\alpha +4H)}
\end{equation}
The inequality (\ref{Eq:Slnd2}) is then obtained by Theorem 1 in \cite{LanMarXiao} for $2< \alpha +4H < 4$.
\end{proof}

As an immediate consequence of Propositions \ref{Prop:Variogram}, \ref{Prop:SLND-t} and 
Proposition \ref{Prop:SLND-x}, we have the following corollary.

\begin{corollary}
Assume $1/2<H<1$ and Conditions \textbf{(A.1)} and \textbf{(A.2)} hold.

\begin{enumerate}
\item[(i)] if $\beta >4H+2$ or $u_{0}\equiv 0,$ then for any $x\in \mathbb{S}^{2}$ fixed, there 
exists a constant $K_{4,6}>1$ such that
\begin{equation*}
K_{4,6}^{-1}[|t-s|^{2\eta}
\leq \mathbb{E}|u( t, x ) -u(s,x ) |^{2} \leq K_{4,6}[|t-s|^{2\eta},
\end{equation*}
where $\eta = H - \max\{(2- \alpha)/4,\, 0\}.$

\item[(ii)] if $u_{0}\equiv 0$ and $\gamma =\alpha /2-1+2H\in (0,1),$ then
for any $t\in \mathbb{T}$ fixed, there exists a constant $K_{4,7}>1$ such
that
\begin{equation*}
K_{4,7}^{-1}\left( d_{\mathbb{S}^{2}}(x,y)\right) ^{2\gamma }
\leq \mathbb{E}|u\left( t,\,x\right) -u\left( t, \, y\right) |^{2}
\leq K_{4,7}\left( d_{\mathbb{S}^{2}}(x,y)\right) ^{2\gamma }.
\end{equation*}
\end{enumerate}
\end{corollary}

\section{Exact Uniform Modulus of Continuity \label{Sec:ModulusContinuity}}

Now we are ready to prove (\ref{eq:Modulus-t}) and (\ref{eq:Modulus-x}) in Theorem 
\ref{Th:Main}. We start by stating a Kolmogorov's 0-1 law regarding the uniform moduli of
continuity for $u$. It is a consequence of Lemma 7.1.1 in Marcus and Rosen \cite%
{MarRos}.

\begin{lemma}
\label{Lem-s4-1} Let $\{u(t),t\in \mathbb{T}\}$ be a centered Gaussian
random process on $\mathbb{T}$, and $\varphi :{\mathbb{R}}_{+}\rightarrow {%
\mathbb{R}}_{+}$ be a function with $\varphi (0+)=0$. If
\begin{equation*}
\lim_{\varepsilon \rightarrow 0}
\sup_{0\leq s<t\leq T,t-s\leq \varepsilon }
\frac{|u(t)-u(s)|}{\varphi (t-s)}\leq K_{5,1},
\hbox{a.s.}
\end{equation*}%
for some constant $K_{5,1}<\infty$, then
\begin{equation*}
\lim_{\varepsilon \rightarrow 0}
\sup_{0\leq s<t\leq T,t-s\leq \varepsilon }
\frac{|u(t)-u(s)|}{\varphi (t-s)}=K_{5,1}^{\prime },
\hbox{a.s.\, for
some constant}\;\;K_{5,1}^{\prime }<\infty .
\end{equation*}
\end{lemma}

We remark that Lemma \ref{Lem-s4-1} does not exclude the possibility of $K_{5,1}^{\prime }=0$. 
One of the main difficulties in establishing an exact uniform modulus of continuity is to find conditions 
under which $K_{5,1}^{\prime }>0$.

\textbf{Proof of (\ref{eq:Modulus-t}) in Theorem \ref{Th:Main}.} The argument is similar to
that in the proof of Theorem 4.1 in \cite{MWX13}. Because
of Lemma \ref{Lem-s4-1}, we see that (\ref{eq:Modulus-t}) in Theorem \ref{Th:Main} will be 
proved after we establish upper and lower bounds of the following form: there exist positive 
and finite constants $K_{5,2}$ and $K_{5,3}$ such that
\begin{equation}
\lim_{\varepsilon \rightarrow 0}
\sup_{0\leq s<t\leq T,t-s\leq \varepsilon }
\frac{|u(t)-u(s)|}
{( t-s) ^{\eta}\sqrt{\big|\log \left( t-s\right) \big|}}
\leq K_{5,2},\ \
\hbox{ a.s. }
\label{Eq:Umod-upper}
\end{equation}%
and
\begin{equation}
\lim_{\varepsilon \rightarrow 0}
\sup_{0\leq s<t\leq T,t-s\leq \varepsilon }
\frac{|u(t)-u(s)|}
{( t-s) ^{\eta}\sqrt{\big|\log \left( t-s\right) \big|}}
\geq K_{5,3},\ \
\hbox{ a.s.}
 \label{Eq:Umod-lower}
\end{equation}%
These and Lemma \ref{Lem-s4-1} with $\varphi (r)=r^{\eta}\sqrt{|\log r|}$ imply
(\ref{eq:Modulus-t}) with $K_{1,1}\in \lbrack K_{5,3},K_{5,2}].$

We divide the rest of the proof into three parts.

\textit{Step 1: Proof of (\ref{Eq:Umod-upper}).} We introduce an auxiliary
Gaussian field:
\begin{equation*}
Y=\{Y(t,s):0\leq s<t\leq T,t-s\leq \varepsilon \}
\end{equation*}%
defined by $Y(t,s)=u(t)-u(s)$, where $0<\varepsilon \leq \delta $ so that
Lemma \ref{Prop:Variogram} holds. By the triangle inequality, we see that
the canonical metric $d_{Y}$ on $\Gamma :=\{(t,s)\in \mathbb{T\times }\mathbb{T}:\,
|t-s|\leq \delta \}$ associated with $Y$ satisfies the following inequality:
\begin{equation}
d_{Y}((t,s),(t^{\prime },s^{\prime }))
\leq \min
\{d_{\mathbb{T}}(t,t^{\prime})+d_{\mathbb{T}}(s,s^{\prime }),
d_{\mathbb{T}}(t,s)+d_{\mathbb{T}}(t^{\prime },s^{\prime })\},  \label{Eq:dY}
\end{equation}%
where $d_{\mathbb{T}}(t,s)=|t-s|^{\eta}.$ Denote the diameter of $\Gamma $ in the 
metric $d_{Y}$ by $D$.  Then, by Lemma \ref{Prop:Variogram}, we have
\begin{equation*}
D\leq \sup_{(s,t),(s^{\prime },t^{\prime })\in \Gamma }
(d_{T}(s,t)+d_{T}(s^{\prime },t^{\prime }))
\leq 2K_{3,1}\,\varepsilon ^{\eta}.
\end{equation*}%
For any $\tau >0$, let $N_{Y}(\Gamma ,\tau )$ be the smallest number of open $d_{Y}$-balls 
of radius $\tau $  needed to cover $\Gamma $. It follows from Lemma \ref{Prop:Variogram} that,
\begin{equation*}
N_{Y}(\Gamma ,\tau )\leq c_{5,1}\tau ^{-\frac{2}{\eta}},
\end{equation*}%
one can verify that
\begin{equation*}
\int_{0}^{D}\sqrt{\ln N_{Y}(T,\tau )}\,d\tau
\leq c_{5,2}\,\varepsilon ^{\eta}\sqrt{\ln (1+\varepsilon ^{-1})}.
\end{equation*}%
Hence, by Theorem 1.3.5 in \cite{AdlTay}, we have
\begin{equation*}
\limsup_{\varepsilon \rightarrow 0}
\sup_{0\leq s<t\leq T,t-s\leq \varepsilon }
\frac{|u(t)-u(s)|}{\varepsilon ^{\eta}\sqrt{|\log \varepsilon |}}
\leq c_{5,3},
\quad \hbox{a.s.},
\end{equation*}%
which implies (\ref{Eq:Umod-upper}). Here, $c_{5,1},$ $c_{5,2}$ and $c_{5,3}$
are positive constants depending only on $c_{0}$, $\alpha $ and $H$.

\textit{Step 2: Proof of (}\ref{Eq:Umod-lower}\textit{).} For any $n\geq\lfloor |\log _{2}\delta |\rfloor +1$, 
where $\delta $ is the constant same
as in Proposition \ref{Prop:Variogram}, we chose a sequence of $2^{n}$
points $\{t_{n,i},1\leq i\leq 2^{n}\}$ on ${\mathbb{T}}$ that are equally
separated in the following sense: For every $2\leq k\leq 2^{n}$, we have
\begin{equation}
t_{n,k}-t_{n,k-1}=2^{-n}.
\label{Eq: Tpoints}
\end{equation}%
Notice that
\begin{equation}
\begin{aligned}
&\lim_{\varepsilon \rightarrow 0}
\sup_{0\leq s<t\leq T,t-s\leq \varepsilon }
\frac{|u(t)-u(s)|}{(t-s)^{\eta}\sqrt{|\log (t-s)|}}
\\
&\geq \underset{n\rightarrow \infty }{\,\lim \inf }
\max_{2\leq k\leq 2^{n}}
\frac{|u(t_{n,k})-u(t_{n,k-1})|}{2^{-n\eta}\sqrt{n}}.
\end{aligned}
\label{Eq:Umod-lower2}
\end{equation}%
It is sufficient to prove that, almost surely, the last limit in (\ref%
{Eq:Umod-lower2}) is bounded below by a positive constant. This is done by
applying the property of strong local nondeterminism in Proposition
 \ref{Prop:SLND-t} and a standard Borel-Cantelli argument.

Let $\tau >0$ be a constant whose value will be chosen later. We consider
the events
\begin{equation*}
A_{m}=\bigg\{
\max_{2\leq k\leq m}\big |u(t_{n,k})-u(t_{n,k-1})\big |
\leq \tau 2^{-nH}\sqrt{n}
\bigg\}
\end{equation*}
for $m=2,\ldots ,2^{n}$. By conditioning on $A_{2^{n}-1}$ first, we can
write
\begin{equation}
\begin{split}
{\mathbb{P}}\big(A_{2^{n}}\big)
& ={\mathbb{P}}\bigg\{\big |u(t_{n,2^{n}})-u(t_{n,2^{n}-1},)%
\big| \leq \tau 2^{-nH}\sqrt{n}\big |A_{2^{n}-1}\bigg\}
\\
&\quad \times {\mathbb{P}}\big(A_{2^{n}-1}\big).
\end{split}
\label{Eq:inter-1}
\end{equation}

Recall that, given the random variables in $A_{2^{n}-1}$, the conditional
distribution of the Gaussian random variable $%
u(t_{n,2^{n}},)-u(t_{n,2^{n}-1})$ is still Gaussian, with the corresponding
conditional mean and variance as its mean and variance. By Proposition \ref{Prop:SLND-t}, 
the aforementioned conditional variance for $k=2,...,2^{n}$
satisfies
\begin{equation*}
\mathrm{Var}\big(u(t_{n,k})-u(t_{n,k-1})\big|A_{k-1}\big)
\geq K_{4,4}2^{-2nH}.
\end{equation*}%
This and Anderson's inequality (see \cite{A55}) imply
\begin{equation}
\begin{split}
& {\mathbb{P}}\bigg\{\big |u(t_{n,k},x_{n,k})-u(t_{n,k-1},x_{n,k-1})\big |
\leq \tau 2^{-n\eta}\sqrt{n}\big |\,A_{k-1}\bigg\} \\
& \leq {\mathbb{P}}\Big\{N(0,1)\leq \frac{2^{-n\eta} \tau \sqrt{n}}{%
K_{4,4}^{1/2}2^{-n\eta}}\Big\} \\
& \leq 1-\frac{K_{4,4}^{1/2}2^{-n\eta}}{\tau 2^{-n\eta}\sqrt{n}}
\exp \Big(-\frac{2^{-2n\eta}C ^{2}n}{K_{4,4}2^{1-2n\eta}}\Big) \\
& \leq \exp \bigg(
-\frac{K_{4,4}^{1/2}}{\tau \sqrt{n}}
\exp \Big(-\frac{\tau^{2}n}{2K_{4,4}}
\Big)\bigg).
\end{split}
\label{Eq:inter-2}
\end{equation}%
In deriving the last two inequalities, we have applied Mill's ratio and the
elementary inequality $1-x\leq e^{-x}$ for $x>0$. Iterating this procedure
in (\ref{Eq:inter-1}) and (\ref{Eq:inter-2}) for $2^{n}-1$ more times, we
obtain%
\begin{equation*}
\begin{split}
{\mathbb{P}}\big(A_{2^{n}}\big)&
\leq \exp \bigg(-\sum_{k=1}^{2^{n}}
\frac{K_{4,4}^{1/2}}{\tau \sqrt{n}}
\exp \Big(-\frac{\tau ^{2}n}{2K_{4,4}}\Big)
\bigg) \\
& \leq \exp \left\{ -\frac{K_{4,4}^{1/2}}{2\tau \sqrt{n}}
\left( \frac{2}{e^{\tau ^{2}/(2K_{4,4})}}\right) ^{n}
\right\} ,
\end{split}%
\end{equation*}%
which yields that
$\sum_{n=1}^{\infty }{\mathbb{P}}\big(A_{2^{n}}\big)<\infty $. Hence the Borel-Cantelli lemma implies that almost surely,
\begin{equation*}
\max_{2\leq k\leq 2^{n}}\big |u(t_{n,k})-u(t_{n,k-1})\big |
\geq \tau 2^{-n\eta}\sqrt{n}
\end{equation*}%
for all $n$ large enough. This implies that the right-hand side of (\ref{Eq:Umod-lower2}) is bounded from below almost surely by some $C >0$.
Hence (\ref{Eq:Umod-lower}) follows from this and Lemma \ref{Lem-s4-1}. This
finishes the proof of (\ref{eq:Modulus-t}) in Theorem \ref{Th:Main}. $\blacksquare $

\textbf{Proof of (\ref{eq:Modulus-x}) in Theorem \ref{Th:Main}. } The proof
is similar to the argument above (see also proof of Theorem 2 in \cite{LanMarXiao}), and we omit it here. $\blacksquare $

As an ending of this section, we further study the regularity properties of higher-order
derivatives of $u$ $w.r.t.$ the spatial variable $x\in \mathbb{S}^{2}$ based on pseudo-differential operators, 
defined as follows: for a real $k\in \mathbb{R}$, 
\begin{equation*}
\nabla^{(k)}u(t):=(1-\Delta _{{\mathbb{S}}^{2}})^{k/2}u(t)=\sum_{\ell m}u_{\ell
m}(t)(1+\ell (\ell +1))^{k/2}Y_{\ell m}
\end{equation*}%
provided the right-hand side converges in $L^{2}(\Omega \times {\mathbb{S}}^{2})$.

It is shown in \cite[Chapter XI]{Taylor} that the Sobolev space $\mathbb{H}^{k}({\mathbb{S}}^{2})$ 
of functions with square-integrable $k$-th weak
derivatives can be viewed as the image of $L^{2}({\mathbb{S}}^{2})$ under
the operator $(1-\Delta _{{\mathbb{S}}^{2}})^{-k/2}$; this and related
property are exploited by Lang et al. \cite{LangSchAAP} and Lan et al. \cite{LanMarXiao} to prove their
results on regularity of higher-order derivatives.  

Again, with a slight abuse of notation, we write $u\left( x\right) =u\left( t, x\right) $ for brevity and recall the estimation 
of the angular power spectrum (\ref{APS-u}) of $u$. An immediate consequence of Theorem 3 
in \cite{LanMarXiao} is as follows, which derives the exact uniform modulus of continuity for $\nabla^{(k)}u(x)$.
\begin{corollary} \label{Cor:derivative-x}
Assume $1/2<H<1$ and Conditions \textbf{(A.1)} and \textbf{(A.2)} hold. If $u_{0}\equiv 0$ and $2k < \alpha+4H< 2+2k$ 
for some integer $k \ge 2$, then $\nabla^{(k-1)}u$ satisfies the following exact uniform modulus of continuity:
\begin{equation*}
\lim_{\varepsilon \rightarrow 0}\sup_{\substack{ x,y\in {\mathbb{S}}^{2}  \\ %
d_{{\mathbb{S}}^{2}}(x,y)\le \varepsilon }}\frac{|\nabla^{(k-1)}u(x)-\nabla^{(k-1)}u(y)|}{%
\left( d_{{\mathbb{S}}^{2}}(x,y)\right)^{\alpha/2+2H-k } \sqrt{\big|\log d_{{%
\mathbb{S}}^{2}}(x,y) \big|}}=K_{5,4}, \quad \hbox{a.s.}
\end{equation*}
\end{corollary}


\end{document}